\documentclass[11pt,a4paper]{article}
\usepackage[margin=1in]{geometry}

\usepackage[utf8]{inputenc}
\usepackage{amsmath}
\usepackage{amsfonts}
\usepackage{amssymb}
\usepackage{pifont}
\usepackage{mathrsfs}
\usepackage{graphicx}
\usepackage{bbm}
\usepackage{float}
\usepackage{amsthm}
\usepackage{tikz}
\usepackage{algorithmic}
\usepackage[linesnumbered,ruled]{algorithm2e}
\usepackage[english]{babel}
\usepackage[babel]{csquotes}
\usepackage[backend=biber,style=numeric, giveninits = true, maxbibnames=99]{biblatex}

\DeclareFieldFormat*{title}{\mkbibemph{#1}}
\DeclareFieldFormat*{citetitle}{\mkbibemph{#1}}
\DeclareFieldFormat{journaltitle}{#1}

\renewbibmacro*{in:}{%
  \ifentrytype{article}
    {}
    {\printtext{\bibstring{in}\intitlepunct}}}

\newbibmacro*{pubinstorg+location+date}[1]{%
  \printlist{#1}%
  \newunit
  \printlist{location}%
  \newunit
  \usebibmacro{date}%
  \newunit}

\renewbibmacro*{publisher+location+date}{\usebibmacro{pubinstorg+location+date}{publisher}}
\renewbibmacro*{institution+location+date}{\usebibmacro{pubinstorg+location+date}{institution}}
\renewbibmacro*{organization+location+date}{\usebibmacro{pubinstorg+location+date}{organization}}

\usepackage[toc,page]{appendix}

\newcommand{\N}{\mathbb{N}}

\newcommand{\Py}{\mathbb{P}}
\newcommand{\Ev}{\mathbb{E}}
\newcommand{\V}{\mathbb{V}}

\newcommand{\cC}{\mathcal{C}}

\newcommand{\cQ}{\mathcal{Q}}

\newcommand{\cO}{\mathcal{O}}
\newcommand{\de}{\partial}

\newcommand{\Cov}{\textnormal{Cov}}

\newcommand{\MSE}{\text{MSE}}
\newcommand{\1}{\mathbbm{1}}

\newcommand{\bs}{\boldsymbol}

\newcommand{\diff}{\textnormal{diff}}

\newcommand{\td}{\text{d}}

\newtheorem{theorem}{Theorem}[section]

\newtheorem{proposition}[theorem]{Proposition}
\newtheorem{corollary}[theorem]{Corollary}

\theoremstyle{remark}

\bibliography{bibl}

\begin{document}

\title{Continuous Level Monte Carlo and Sample-Adaptive Model Hierarchies}
\author{Gianluca Detommaso${}^{1,2}$, Tim Dodwell${}^3$ and Rob Scheichl${}^2$}
\date{}
\maketitle

\begin{center}
\begin{footnotesize}

\vspace{-0.75cm}

\noindent
${}^1$ The Alan Turing Institute, London, NW1 2DB, UK. \ Email: {\tt gdetommaso@turing.ac.uk}

\vspace{0.1cm}

\noindent
${}^2$ Department of Mathematical Sciences, University of Bath, Bath, BA2
7AY, UK. 

\vspace{0.1cm}

\noindent
${}^3$ College of Engineering, Mathematics and Physical Sciences,
University of Exeter, Exeter, EX4 4PY, UK. 

\vspace{0.1cm}

\end{footnotesize}
\end{center}


\begin{abstract}
In this paper, we present a generalisation of the 
Multilevel Monte Carlo (MLMC) method to a setting where the level
parameter is a continuous variable.
This Continuous Level Monte Carlo
(CLMC) estimator provides a natural framework in PDE applications 
to adapt the model hierarchy to each sample. In addition, it can be made 
unbiased with respect to the expected value of the true quantity of
interest provided the quantity of interest converges sufficiently
fast. The practical implementation of the CLMC
estimator is based on interpolating actual evaluations of the quantity of
interest at a finite number of resolutions. As our
new level parameter,
we use the logarithm of a goal-oriented finite element error estimator
for the accuracy of the quantity of interest. We prove the
unbiasedness, as well as a complexity
theorem that shows the same rate of complexity for CLMC as for
MLMC. Finally, we
provide some numerical evidence to support our theoretical results, by
successfully testing CLMC on a standard PDE test problem. The
numerical experiments demonstrate clear gains for sample-wise adaptive
refinement strategies over uniform
refinements.
\end{abstract}

\section{Introduction}
No matter whether epistemic or aleatoric, known unknown or unknown
unknown, uncertainty plays a fundamental role in any real life
situation. Its quantification is becoming an object of interest for
ever more complex problems, where accurate solutions require huge
computational costs.
A lot of methods have been proposed in the last decade that aim to reduce this cost without affecting the accuracy. Among others, multilevel techniques conquered the scene arising in a multitude of algorithms, all following the pioneering work on \textit{multilevel Monte Carlo} (MLMC) by Giles \cite{giles2008multilevel} and the earlier paper by Heinrich \cite{heinrich2001multilevel} (see also \cite{elfverson2016multilevel, cliffe2011multilevel} and references therein). In general, multilevel techniques aim to accelerate inference by exploiting a hierarchy of models with different levels of accuracy. By combining estimates from all the models in a telescoping sum, the computational cost is shifted towards the bottom (cheap and inaccurate) end of the hierarchy, while maintaining the accuracy of the top (expensive and high resolution) end.

Since the initial work on MLMC, several techniques have been employed to exploit model structures even further with considerable savings in computational cost. An important step forward was the introduction of \textit{adaptive multilevel Monte Carlo} (AMLMC) \cite{hoel2012adaptive}, 
where error estimates and adaptive refinement strategies are exploited
to increase the accuracy only where needed (see also 
\cite{elfverson2016multilevel,eigel2016adaptive,babushkina2016adaptive}
in the context of PDEs). In contrast to the majority of the literature
on MLMC, which is based on uniform refinements, AMLMC is able to deal
with problems with very localised sample-dependent noise or
quantities of interest, avoiding excessive computational cost by
refining the models only where necessary and, in general, differently
for each sample. 

A second important step forward was the introduction of an MLMC
estimator that is unbiased with respect to the real quantity of
interest \cite{rhee2015unbiased} (see also \cite{mcleish2011general,
  vihola2015unbiased}). In most problems of consideration, the
quantity of interest is a functional of the solution of an
inaccessible, infinite-dimensional model. In such cases, standard MLMC
is only able to provide an estimator that is unbiased with respect to
an approximation of the real quantity of interest. Having an unbiased
estimator for the real quantity of interest is often of great
practical interest, especially if the estimator is used for further
predictions. Furthermore, the bias error is typically harder to estimate than the
sampling error, making it easier to avoid unnecessary computational 
effort with an unbiased estimator.

In this paper, we present a generalisation of MLMC to a continuous
framework that we denote \textit{continuous level Monte Carlo} (CLMC),
where the underlying hierarchical structure is considered to be
continuous rather than a finite sequence of discrete instances. The
level parameter $\ell$ is assumed to be a real number rather than an
integer, giving access to standard tools from Calculus, such as the
integral or the derivative with respect to the level. Although this
might sound just like a conceptual generalisation, we will
interestingly see how the continuous framework also allows deeper
understanding and different perspectives. As a first fact, it
highlights a link with tools from probability theory, since the
continuous sequence of approximations can now be interpreted as a
continuous stochastic process over the level of resolution. In this
framework, the classic telescoping sum of MLMC straightforwardly
becomes a simplified version of Dynkin's formula
\cite{kesendal2000stochastic}, or more simply the Fundamental Theorem
of Calculus. As allowed in Dynkin's formula, the finest level $L$ of
resolution can be chosen as a stopping time random variable, which
stops the refining procedure differently for each sample according to
some probability distribution over $L$. We will see that there is a
simple probability distribution over $L$ corresponding to the optimal
decaying sequence of the number of samples in MLMC and the choice of
this distribution is not very sensitive to an accurate estimation of
the convergence rates and of the cost per sample.

The main result of the paper is a continuous version of the complexity theorem for MLMC. This provides two main contributions:
\begin{itemize}
\item it introduces a CLMC estimator that, under standard assumptions, satisfies the same computational cost rate as the one in MLMC;
\item it proves that the CLMC estimator can potentially be unbiased, \textit{but} the unbiased version has finite computational cost exclusively when the variance decays faster than the cost per sample grows.
\end{itemize}
Among potential applications, the continuous level framework finds his
practical utility for sample-dependent hierarchical refinements: when
the refinement levels depend on samples instead of being fixed, it is
more natural to think of them in a continuous fashion, as the
resolution of a particular model can fall anywhere on the real
line. This is a typical situation in AMLMC. Indeed, the resolution
level is usually interpreted as the logarithm of the error of the
numerical model, therefore intrinsically continuous. Moreover, the
error is sample-dependent, hence each sample will hit its own sequence
of level refinements. As AMLMC involves taking sample averages of
quantities of interests at some prescribed levels, approximations have
to be made that may lead to slight inefficiencies especially when the
improvement in the approximation error in each adaptation step varies 
strongly (see \cite{babushkina2016adaptive}).
 
Here, we develop practical CLMC algorithms that are easy to implement
and do not require any such approximation. As we can arbitrarily
choose the nature of the quantity of interest between the actual
evaluations, to obtain a quantity of interest function that is
continuous over the levels we simply interpolate
the calculated values, whence we can work out a practical
formula. Note that the practical formula can also be implemented for
the unbiased version
of the CLMC estimator. Finally, we provide some numerical experiments 
showing the CLMC algorithm in action for a standard two-dimensional
model problem where the adaptivity and the sample-dependent hierarchies are
shown to leading to significant computational savings.

The structure of the paper is as follows. In Section \ref{sec:CLMC},
we give a short background of Monte Carlo and MLMC; we present the
main CLMC idea; we introduce the CLMC estimator and show the
unbiasedness property; we state the CLMC complexity theorem; we
provide a corollary showing when the estimator that provides the
optimal cost is unbiased with respect to the real quantity of interest.
In Section \ref{sec:practical}, we propose a practical CLMC algorithm
for sample-based adaptive hierarchical refinement; we discuss the
special case of uniform refinement and the similarities with MLMC and
show the link between the distribution of the finest level and the
sequence of number of samples; we finish the section with
some proposals for other possible implementations and approaches.
Finally, Section \ref{sec:AMLMC} introduces the PDE model problem and
the adaptive finite element hierarchy for them, as well as presenting
and discussing the numerical experiments. We finish the paper with
some conclusions and ideas for future work in Section
\ref{sec:conclusion}. The detailed proof of the complexity theorem, as
well as some details about the goal-oriented error estimator are
delegated to the appendices.

\section{Continuous level Monte Carlo}\label{sec:CLMC}

\subsection{Background: Monte Carlo and Multilevel Monte Carlo}
Suppose one is interested in estimating the expected value $\Ev[\cQ]$ of some (inaccessible) quantity of interest $\cQ$, for simplicity assumed to be scalar. In uncertainty quantification (UQ), $\cQ$ is typically a functional of the solution of some random partial differential equation (PDE), where the randomness can lie anywhere, e.g. within the coefficients, the source, the boundary conditions or the shape of the domain.

In general, the solution of a PDE can not be calculated exactly and it
has to be approximated numerically, up to some desirable resolution
level $L$. Let us call $Q_L$ such an approximation and assume that
$Q_L\to \cQ$ almost surely (a.s.) for $L\to +\infty$. Then, for any desired tolerance $\varepsilon > 0$, there exists a fine enough resolution $L$, such that $|\Ev[\cQ-Q_L]| \le \varepsilon$, and we can focus on finding good algorithms to estimate $\Ev[Q_L]$ to the same accuracy. There are two main issues here.
\begin{enumerate}
\item If the underlying probability distribution is continuous and high-dimensional, which is common in UQ applications, it can be extremely expensive to approximate the expected value with standard quadrature methods.
\item If the resolution $L$ required to compute the PDE solution with sufficient accuracy is high, then computing just one sample of $Q_L$ will be expensive and the number of samples that can be computed on level $L$ in a reasonable time is limited.
\end{enumerate}
A standard remedy for Issue 1 is the use of \textit{Monte Carlo} (MC) methods \cite{robert2004monte}. Indeed, the rate of converge of MC estimators is independent of the dimension of the integral and it is extremely easy to implement: given $N$ independent samples $\big(Q_L^{(k)}\big)_{k=1}^{N}$ of $Q_L$, distributed according to the underlying probability distribution, the expected value can be estimated as
\begin{equation}\label{MCest}
\Ev[Q_L]\approx \frac{1}{N}\sum_{k=1}^NQ_L^{(k)}.
\end{equation}

Whilst the right-hand-side in \eqref{MCest} is an unbiased estimator
of $\Ev[Q_L]$, unfortunately it converges very slow, especially when
$L$ is large, since $\cO(\varepsilon^{-2})$ samples are required to
reduce the sampling error to a given accuracy $\varepsilon$,
i.e. $|\Ev[\cQ-Q_L]|\le\varepsilon$. As every sample requires an
expensive PDE solve, the computational cost quickly becomes infeasible
for small $\varepsilon$.

An acceleration technique suggested for \eqref{MCest} is the {\em multilevel Monte Carlo} (MLMC) method \cite{heinrich2001multilevel,giles2008multilevel}. It exploits a hierarchy of approximations $Q_0,Q_1,\dots,Q_L$ of $\cQ$ at different resolutions, starting with a coarse and cheap approximation $Q_0$, and going up to the fine and expensive approximation $Q_L$. In contrast to the standard MC estimator in \eqref{MCest}, which directly estimates $\Ev[Q_L]$ by sampling $Q_L$, MLMC combines samples from the sequence of approximations $(Q_\ell)_{\ell=0}^L$ to produce an overall cheaper estimator. To this purpose, the approximations are combined into the telescoping sum   
\begin{equation}\label{tel_sum}
\Ev[Q_L-Q_0]=\sum_{\ell=1}^L\Ev[Q_\ell-Q_{\ell-1}],
\end{equation}
and then each term in the sum on the right-hand-side is estimated with Monte Carlo:
\begin{equation}\label{MLMCest}
\Ev[Q_\ell-Q_{\ell-1}] \approx \frac{1}{N_\ell}\sum_{k=1}^{N_\ell} \left(Q_\ell^{(k)}-Q_{\ell-1}^{(k)}\right).
\end{equation}
To obtain an estimator for $\Ev[Q_L]$ it suffices to add a Monte Carlo estimator for $\Ev[Q_0]$.
 
Crucially, the consecutive approximations $Q_{\ell-1}^{(k)}$ and $Q_\ell^{(k)}$ in the difference $Q_\ell^{(k)}-Q_{\ell-1}^{(k)}$ come from the same sample $k$. This means that they are strongly positively correlated, and the variance of the difference is heavily reduced:
\begin{equation}\label{ML_var_red}
\V[Q_\ell-Q_{\ell-1}] = \V[Q_{\ell-1}]+\V[Q_\ell]-2\Cov(Q_{\ell-1},Q_\ell) \ll \V[Q_{\ell-1}]+\V[Q_\ell].
\end{equation}
As $Q_\ell\to \cQ$ a.s. for $\ell\to+\infty$, we also have $Q_\ell - Q_{\ell-1} \to 0$, so that the covariance, and in turn the variance reduction, increases as $\ell\to+\infty$. As a consequence, the required number of samples $N_\ell$ at level $\ell$ can be chosen to decrease monotonically with increasing $\ell$, so that only very few expensive samples on level $L$ are needed. The majority of samples and therefore the computational cost will be shifted to the coarser levels.

This reduction in computational complexity can be quantified rigorously, at least asymptotically as the tolerance $\varepsilon \to 0$. The complexity theorems in \cite{giles2008multilevel, cliffe2011multilevel} show that the overall computational cost for the MLMC algorithm can be up to a factor $\cO(\varepsilon^{2})$ smaller than the cost of the MC estimator in \eqref{MCest}. We will return to this and give more details in Section \ref{sec_compl_thm}. 

\subsection{Continuous Level Monte Carlo: the main idea}

In this section, we introduce the {\em continuous level Monte Carlo} (CLMC) idea. As we have seen above, MLMC exploits a discrete sequence of approximations $(Q_\ell)_{\ell=0}^L$ of $\cQ$. We now extend this to a continuous family of approximations $(Q(\ell))_{\ell\ge 0}$ of $\cQ$. In other words, $(Q(\ell))_{\ell\ge 0}$ is a stochastic process of approximations over the continuous level of resolution $\ell$.

Let $L$ be assumed to be a random variable with finite expectation denoting the (random) finest level of resolution, independent from the stochastic process $(Q(\ell))_{\ell\ge 0}$. Also, let $L_{\max}\in [0,\infty]$ be a deterministic constant that we introduce for reasons that will become clearer later. We can write down the following formula:
\begin{equation}\label{CLtel_sum}
\Ev[Q(L\wedge L_{\max}) - Q(0)] = \Ev\left[\int_0^{L\wedge L_{\max}} \frac{\td Q(\ell)}{\td\ell}\,\td\ell\right]\,.
\end{equation}
For the formula in \eqref{CLtel_sum} to be well-posed, we need to assume that $Q\in W^{1,1}(0,L_{\max})$ as a function of $\ell$, where $W^{1,1}(0,L_{\max})$ is a Sobolev space containing functions over $\ell\in(0,L_{\max})$ such that the functions and their weak first derivatives have finite $L^1$ norm. Note that for simplicity we are choosing 0 as coarsest level, but this can of course be generalised.

If we assume $L$ to be a deterministic variable, the expectation in \eqref{CLtel_sum} can be pulled inside the integral and the derivative, so that equation \eqref{CLtel_sum} reduces to the Fundamental Theorem of Calculus, which guarantees the identity. However, more generally, equation \eqref{CLtel_sum} can be recovered as a particular case of Dynkin's Formula \cite{kesendal2000stochastic}, where $L$ is interpreted as a finite stopping time.

\subsection{The CLMC estimator}
Let us assume $L$ to be a random variable independent of the whole stochastic process $(Q(\ell))_{\ell\ge 0}$. We can then define the \textit{continuous level Monte Carlo (CLMC) estimator}
\begin{equation}\label{CLest}
\widehat{Q}^{\text{CLMC}}_{L_{\max}} := \frac{1}{N}\sum_{k = 1}^{N}\int_0^{L_{\max}}\frac{1}{\Py(L\ge \ell)}\left(\frac{\td Q}{\td\ell}\right)^{(k)}(\ell)\,\1_{[0,L^{(k)}]}(\ell)\,\td\ell,
\end{equation}
where the superscript $(k)$ denotes the $k$-th realisation of the
respective random variable and $N$ is the total number of samples. For
simplicity of presentation, the estimator
$\widehat{Q}^{\text{CLMC}}_{L_{\max}}$ is defined as an estimator for
$\Ev[Q(L_{\max})-Q(0)]$, as we see in Proposition
\ref{prop_unbiased}. As in standard MLMC, it suffices to add an
unbiased estimator for $\Ev[Q(0)]$ to obtain an estimator for $\Ev[Q(L_{\max})]$.

A reader familiar with the MLMC literature might be puzzled by the
estimator in \eqref{CLest}, where we use the same number of samples
$N$ for each level $\ell$. However, note that, for each sample $k$,
the integrand in \eqref{CLest} will only be non-zero up to the random
realisation $L^{(k)}$ of $L$, and therefore in practice we do not need
to evaluate $Q(\ell)$ beyond level $L^{(k)}$.

We are now ready to show that the CLMC estimator is unbiased.

\begin{proposition}\label{prop_unbiased}
The CLMC estimator \eqref{CLest} is an unbiased estimator for $\Ev[Q(L_{\max})-Q(0)]$, i.e.
\[
\Ev[\widehat{Q}^{\textnormal{CLMC}}_{L_{\max}}] = \Ev[Q(L_{\max})-Q(0)]\,.
\]
\end{proposition}
\begin{proof}
By exploiting the independence of $L$ from $(Q(\ell))_{\ell\ge 0}$, we have
\begin{align*}
\Ev\left[\widehat{Q}^{\text{CLMC}}_{L_{\max}}\right] &= \Ev\left[\frac{1}{N}\sum_{k = 1}^{N}\int_0^{{L_{\max}}}\frac{1}{\Py(L\ge \ell)}\left(\frac{\td Q(\ell)}{\td\ell}\right)^{(k)}\1_{[0,L^{(k)}]}(\ell)\,\td\ell\right]\\
&= \int_0^{{L_{\max}}}\frac{1}{\Py(L\ge \ell)}\Ev\left[\frac{\td Q(\ell)}{\td\ell}\right]\Ev\left[\1_{[0,L]}(\ell)\right]\,\td\ell\\
&= \int_0^{{L_{\max}}}\frac{1}{\Py(L\ge \ell)}\Ev\left[\frac{\td Q}{\td\ell}(\ell)\right]\Py(L\ge\ell)\,\td\ell\\
&= \int_0^{L_{\max}} \Ev\left[\frac{\td Q(\ell)}{\td\ell}\right]\,\td\ell\\
&= \Ev[Q(L_{\max}) - Q(0)].
\end{align*}
\end{proof}

In particular, this implies the following important corollary.
\begin{corollary}\label{coroll_unbiased}
If $L_{\max}=+\infty$, then 
\[ \Ev[\widehat{Q}^{\textnormal{CLMC}}_{\infty}] = \Ev[\cQ-Q(0)]. \]
\end{corollary}

Corollary \ref{coroll_unbiased} shows that there is a version of the estimator \eqref{CLest} that is unbiased with respect to the expectation of the difference of the real quantity of interest $\cQ$ and $Q(0)$, and one can see the connection with the unbiased MLMC estimator introduced in \cite{rhee2015unbiased}.

In the next subsection, we will prove a complexity theorem for the
CLMC estimator \eqref{CLest}. We will pick $L$ to be distributed as an
exponential random variable to facilitate calculations and mimic the
exponential decay in the assumptions on the convergence of the
quantity of interest. Also, we will provide sufficient
and necessary conditions for the Theorem to hold in the case
$L_{\max}=+\infty$, i.e. when the CLMC estimator is unbiased with
respect to $\cQ-Q(0)$. A practical algorithm will then be described
 in Section \ref{sec:practical}.

\subsection{Complexity theorem}\label{sec_compl_thm}

The fundamental theoretical result about the MLMC method is the
complexity theorem, firstly proved in \cite{giles2008multilevel} and
generalised in \cite{cliffe2011multilevel}. In this section, we state an analogous complexity theorem for the CLMC estimator \eqref{CLest}. A full proof is given in Appendix \ref{app_CT}.

First, let us define the mean-squared-error (MSE) of the CLMC estimator $\widehat{Q}^{\text{CLMC}}_{L_{\max}}$ in
\eqref{CLest}  by
\begin{equation}\label{MSEbound}
\MSE :=\Ev\left[\big(\widehat{Q}^{\text{CLMC}}_{L_{\max}}-\Ev[\cQ-Q(0)]\big)^2\right]
\end{equation}
and denote by $\cC^{\textnormal{CLMC}}_{L_{\max}}$ its
expected computational cost.  Then, we have the following result.
\begin{theorem}[Complexity Theorem] \label{thm_CLMC_giles}
Suppose $\cQ$ is a quantity of interest and $Q\in W^{1,1}(0,\allowbreak \infty)$ is a corresponding family of numerical approximations.
Furthermore, suppose that there are positive constants $\alpha,\
\beta\le 2\alpha,\ \gamma,\ c_1,\ c_2,\ c_3$ such that, for any
$\ell>0$, we have: 
\begin{itemize}
\item[(i)] $\left|\Ev\left[\frac{dQ(\ell)}{d\ell}\right]\right|\le c_1 e^{-\alpha \ell}\,$,
\qquad (ii) \ $\V\left[\frac{dQ(\ell)}{d\ell}\right]\le c_2e^{-\beta\ell}\,$,
\item[(iii)] $\cC(\ell)\le c_3 e^{\gamma\ell}\,$, where $\cC(\ell)$ is the cost to compute one sample of $Q(\ell)$.
\end{itemize}
Furthermore, suppose that 
$L\sim\textnormal{Exponential}(r)$ with 
\[
r\in[\min(\beta,\gamma),\ \max(\beta,\gamma)]. 
\] 
Then, for any $\varepsilon \in (0,e^{-1})$, there exist $L_{\max}\in [0,+\infty)$,
$N \in \mathbb{N}$ and $C>0$ such that 
\begin{equation}
\label{eq:complex}
\MSE \; \le \; \varepsilon^2\quad\textnormal{and}\quad\cC^{\textnormal{CLMC}}_{L_{\max}}
\;\le \;C  \, \varepsilon^{-2 - \max(0,\frac{\gamma-\beta}{\alpha})}(\log\varepsilon)^{\delta_{r,\beta}+\delta_{r,\gamma}}
\end{equation}
with $\delta$ denoting the Kronecker delta.
\end{theorem}

Note that the predicted computational cost in Theorem
\ref{thm_CLMC_giles} is 
the same as in MLMC (asymptotically).

\begin{corollary}\label{coroll_thm}
Suppose that the assumptions of Theorem \ref{thm_CLMC_giles} hold and that $L_{\max} = +\infty$, i.e. let us consider the unbiased CLMC estimator 
$\widehat{Q}^{\text{CLMC}}_{\infty}$.
\begin{enumerate}
\item[(a)] If $\beta > \gamma$, then for any $\varepsilon \in (0,e^{-1})$ and for
  any $r \in (\gamma,\beta)$, there exists an $N \in \mathbb{N}$ and
  $C>0$ such that
\[
\MSE \le
\varepsilon^2\qquad\textnormal{and}\qquad\cC^{\textnormal{CLMC}}_{\infty}
\le C\varepsilon^{-2} \,.
\]
\item[(b)] If $\beta \le \gamma$ and, in addition, there exist positive constants $\eta \in [\beta,\gamma]$, $c_2'$ and $c_3'$ such that 
\[
c_2' e^{-\eta\ell} \le \V\left[\frac{dQ(\ell)}{d\ell}\right] \quad \text{and} \quad 
c_3' e^{\eta\ell} \le \cC(\ell)\,,
\]
then $\MSE \,\times \,\cC^{\textnormal{CLMC}}_{\infty} = +\infty$,  for all $r > 0$ and $N \in \mathbb{N}$, i.e. the unbiased estimator has infinite MSE or infinite cost.
\end{enumerate}
\end{corollary}

Corollary \ref{coroll_thm} provides sufficient and necessary
conditions for the CLMC estimator with $L_{\max}=+\infty$ (which is
unbiased with respect to $\cQ-Q(0)$) to have a finite expected
complexity cost. Intuitively, since $L_{\max}=+\infty$, the finest
level at which computations are needed is $\max_{k=1}^N L^{(k)}$,
which tends to infinity as $N$ grows. Therefore, the estimator
\eqref{CLest} will have finite expected cost only if the actual
variance reduction rate is bigger than the actual cost growth
rate. The rates $\beta$ and $\gamma$ in Theorem \ref{thm_CLMC_giles}
are only upper bounds. By analogy, we believe this constraint also
applies to the unbiased estimator introduced by Rhee \& Glynn
\cite{rhee2015unbiased}. However, the paper \cite{rhee2015unbiased} is
mainly concerned with timestepping methods for SDEs, where the
condition $\gamma<\beta$ is usually satisfied.

Note that, if $L_{\max}=+\infty$, even in the case $\beta > \gamma$,
there is a non-zero probability that the finest level $L^{(k)}$ for
some sample $(k)$ is drawn larger than the maximal refinement
achievable on the particular machine that is used, but we can exactly
quantify the probability for this to happen. Indeed, if $\bar{L}$ is
the maximum refinement level achievable by the machine, the
probability that at least one sample is greater or equal than $\bar{L}$ is given by
\[ N\Py(L\ge \bar{L}) = N\exp(-r\bar{L}). \]
We will see that for problems of interests this probability is very
small. In the rare event that $L^{(k)} > \bar{L}$ for some sample $k$,
one could simply approximate $Q^{(k)}(\ell) = Q^{(k)}(\bar{L})$ for
$\ell\in[\bar{L},L^{(k)}]$. If $\bar{L}$ is sufficiently large, this
would introduce a negligible bias error to any practical values of
$\varepsilon$.

\section{Practical implementation}
\label{sec:practical}

In the previous section, we have seen that it is possible to extend multilevel Monte Carlo to a continuous framework, where the approximations of the quantity of interest are functions over a continuous family of resolutions. This point of view comes natural when the level parameter is not associated with some fixed hierarchy of approximations, but with an adaptively chosen hierarchy for each sample, e.g. in the context of adaptive finite element approximations of a PDE with random coefficients where the level parameter $\ell$ is related to the accuracy of the approximation (see Section \ref{sec:AMLMC}).

However, it still remains to show how this can be implemented in practice and how the practical implementation differs from MLMC. 
There are many possible ways to implement the estimator in \eqref{CLest}. Let us first focus in some sense on the simplest one. We will comment on other approaches at the end of this section.

\subsection{Sample-dependent level hierarchies and piecewise linear interpolation}\label{subsec:pwlinear}

Let us assume that we have estimates of the parameters $\alpha, \beta, \gamma$ in Theorem \ref{thm_CLMC_giles}. In practice, these can be obtained (on the fly) from sample averages and sample variances of $Q(\ell)$ and $\td Q(\ell)/\td\ell$, as in standard MLMC. Then, given a desired tolerance $\varepsilon>0$, Theorem \ref{thm_CLMC_giles} provides suitable choices for the number of samples $N$ and for the rate $r$ of the exponential distribution of $L$ to achieve the optimal complexity in \eqref{eq:complex}.

For any sample $k$, suppose that $(Q_j^{(k)})_{j\ge 1}$ denotes a countable sequence of approximations of $Q^{(k)}$ at levels $(\ell_j^{(k)})_{j\ge 1}$. Then, to define a continuous family $Q^{(k)}(\ell)$ of $Q^{(k)}$, we use linear interpolation such that
\[
\left(\frac{\td Q}{\td\ell}\right)^{(k)}(\ell) := 
\frac{Q_{j}^{(k)} - Q_{j-1}^{(k)}}{\ell^{(k)}_{j} - \ell^{(k)}_{j-1}}\quad\quad \textnormal{for }\ell\in (\ell_{j-1}^{(k)},\ell_j^{(k)})\,.
\]
Also, for each sample $k$, let us define the index $J^{(k)}$ corresponding to the first value of $\ell_j^{(k)}$ that is bigger than $L^{(k)}\wedge L_{\max}$, that is
\[ J^{(k)} := \min\{j\ge 1:\ \ell_j^{(k)}-(L^{(k)}\wedge L_{\max}) \ge 0 \}.\]
Hence, we can write down the CLMC estimator \eqref{CLest} as
\begin{align}
\widehat{Q}^{\textnormal{CLMC}}_{L_{\max}} & = \frac{1}{N}\sum_{k = 1}^{N}\int_{0}^{L^{(k)}\wedge L_{\max}} \frac{1}{P(L\ge \ell)}\left(\frac{\td Q}{\td\ell}\right)^{(k)}(\ell) \,\td\ell \nonumber\\[1ex]
& = \frac{1}{N}\sum_{k = 1}^{N}\sum_{j=1}^{J^{(k)}} w_j^{(k)}\left(Q_{j}^{(k)} - Q_{j-1}^{(k)}\right) \,, \label{CLMCest_pract}
\end{align}
where we define 
\begin{equation}\label{elltilde}
\tilde{\ell}_j^{(k)} := \ell_{j}^{(k)}\wedge (L^{(k)}\wedge L_{\max})\,,
\end{equation}
and the integrals in the weights $w_j^{(k)}$ can be computed explicitly as
\begin{equation}\label{CLMCest_w}
w_j^{(k)} := \frac{1}{\ell^{(k)}_{j} - \ell^{(k)}_{j-1}} 
\int_{\ell_{j-1}^{(k)}}^{\tilde{\ell}_j^{(k)}} \frac{1}{P(L\ge \ell)}
\,\td\ell =
\frac{\exp\big(r \tilde{\ell}_j^{(k)}\big) - \exp\big(r
  \ell_{j-1}^{(k)}\big)}{r\big(\ell^{(k)}_{j} - \ell^{(k)}_{j-1}\big)}\,,
\end{equation}
for all $j=1,\ldots,J^{(k)}$\,.
Algorithm \ref{alg_CLMC} provides the key instructions to implement the CLMC estimator in \eqref{CLMCest_pract}.\\[10pt]
{\centering
\begin{minipage}{.9\linewidth}
\begin{algorithm}[H]
\caption{CLMC algorithm -- Key steps}\label{alg_CLMC}
\SetKwInOut{Input}{Input}
\SetKwInOut{Output}{Output}
\Input{$\varepsilon$: tolerance;\\
$r$: exponential rate;\\
$N$: total number of samples;\\
$L_{\max}$: maximum reachable level - potentially infinite if $\gamma<\beta$.\vspace{1.5ex}}
\Output{$\widehat{Q}^{\textnormal{CLMC}}_{L_{\max}}$: CLMC estimator.\vspace{1.5ex}}
\begin{algorithmic}[1]
\STATE Initialise $\hat{Q}\gets 0$;\vspace{1.5ex}
\FOR{$k = 1,2,\dots, N$}
\STATE Sample $L^{(k)}\sim\textnormal{Exponential}(r)$;
\STATE Evaluate and store quantity of interests $Q\gets (Q_j^{(k)})_{j=1}^{J^{(k)}}$ at levels $\ell \gets (\ell_j^{(k)})_{j=1}^{J^{(k)}}$;
\STATE Calculate array $w$ of weights in \eqref{CLMCest_w};
\STATE Update $\widehat{Q}^{\textnormal{CLMC}}_{L_{\max}} \gets \widehat{Q}^{\textnormal{CLMC}}_{L_{\max}} + w^T*\diff(Q)$, where $\diff(Q)$ is the array of the differences between consecutive elements of $Q$;
\ENDFOR \vspace{1.5ex}
\STATE Set $\widehat{Q}^{\textnormal{CLMC}}_{L_{\max}} \gets \widehat{Q}^{\textnormal{CLMC}}_{L_{\max}}/N$.
\end{algorithmic}
\end{algorithm}
~\\[10pt]
\end{minipage} 
}

Note that it is easy to work out an unbiased estimator for the variance of $\widehat{Q}^{\textnormal{CLMC}}_{L_{\max}}$ in \eqref{CLMCest_pract}, which is needed to estimate the total number of samples $N$. Let us define
\[ Y^{(k)} := \sum_{j=1}^{J^{(k)}} w_j^{(k)}\left(Q_{j}^{(k)} - Q_{j-1}^{(k)}\right)\,. \]
Then \eqref{CLMCest_pract} simply reduces to a standard Monte Carlo
estimator with i.i.d. samples $Y^{(k)}$ and we can estimate
\[ \V\left[\widehat{Q}^{\textnormal{CLMC}}_{L_{\max}}\right] 
\approx \frac{1}{N(N-1)}\sum_{k=1}^N \left(\,\left(Y^{(k)}\right)^2 - \left(\frac{1}{N}\sum_{i=1}^N Y^{(i)}\right)^2\, \right)\,.  \]

\subsection{Uniform refinements as a special case}

It is interesting to see what happens in the case of uniform
refinements, where all samples $Q^{(k)}$, for $k=1,\ldots,N$, are
evaluated at the same deterministic points $\ell_j^{(k)} = \ell_j$,
for $j\ge 1$, and then interpolated. Without loss of generality, we assume that $\ell_j = j$, as in standard MLMC.

In this case, the set of possible levels reduces to
integers. Therefore, although a continuous probability distribution
for $L$ is still a valid choice, it is more natural to pick a discrete
distribution over the levels, where $\Py(L\ge j)$ is constant over the
interval $(j-1,j)$. In that case, the practical CLMC estimator in \eqref{CLMCest_pract} reduces to
\[ \widehat{Q}^{\textnormal{CLMC}}_{L_{\max}} = \frac{1}{N}\sum_{k=1}^N\sum_{j=1}^{J^{(k)}} \frac{1}{\Py(L\ge j)}\left(Q_{j}^{(k)} - Q_{j-1}^{(k)}\right)\,.  \]
A natural choice would be a geometric distribution on $L$. 

To see the relationship with the standard MLMC estimator more clearly, let us define 
\[ N(\ell) := N \Py(L \ge \ell) \,. \]
Then, $(N(\ell))_{\ell\ge 0}\subset [0,\infty)$ corresponds to a continuous density of samples, analogous to the sequence of sample sizes at discrete levels in MLMC. Moreover, the probability that $L$ is at least $\ell$ corresponds to the normalised density of samples that gets at least to level $\ell$.
Therefore, by plugging this relation in the equation above, we get
\[ \widehat{Q}^{\textnormal{CLMC}}_{L_{\max}} = \sum_{k=1}^N\sum_{j=1}^{J^{(k)}} \frac{1}{N(\ell)}\left(Q_{j}^{(k)} - Q_{j-1}^{(k)}\right)\,,  \]
which exactly corresponds to the Rhee \& Glynn estimator in \cite{rhee2015unbiased}.

\subsection{Other Implementations}
\subsubsection{Polynomial regression}
Although the practical implementation discussed in Subsection \ref{subsec:pwlinear} is a natural, practical implementation of the CLMC estimator, it is not the only possibility. One could think of exploiting the underlying continuous level structure in order to predict the global trend of the function $Q(\ell)$, thereby denoising the point-wise evaluations coming from the random samples. More concretely, imagine that each sample $k$ provides evaluations $(Q_j^{(k)})_{j=1}^{J^{(k)}}$ respectively at levels $(\ell_j^{(k)})_{j=1}^{J^{(k)}}$. Instead of defining the function $Q^{(k)}(\ell)$ as the linear interpolant between the given points as in Subsection \ref{subsec:pwlinear}, one could define $Q^{(k)}(\ell)$ to be a particular polynomial interpolant or regression function. The resulting continuous function may not exactly interpolate the points but rather catch the global trend, avoiding to overfit sample-dependent noisy oscillations.

In general, for each sample $k$, define the polynomials
\[
\left(\frac{\td Q}{\td\ell}\right)^{(k)}(\ell) := 
\sum_{i=0}^{n_p-1}a_{ij}^{(k)}\ell^i\quad\quad \textnormal{for }\ell\in [\ell_{j-1}^{(k)},\ell_j^{(k)})\,,
\]
where the coefficient $(a_{ij}^{(k)})_{i=0}^{n_p-1}$ come from some $n_p$-order polynomial regression procedure, for $j=1,\dots,J^{(k)}$. As in standard MLMC, one needs to make sure that the consecutive increments cancel properly; therefore, the fit procedure must be such that the polynomials $Q^{(k)}(\ell)$ coincide at the interval extremes $(\ell_j^{(k)})_{j=2}^{J^{(k)}-1}$, i.e. $Q^{(k)}(\ell)$ is a continuous function.

As in Subsection
\ref{subsec:pwlinear}, it can be shown that the resulting CLMC estimator is given by
\begin{equation}\label{CLMCest_pract2}
\widehat{Q}^{\textnormal{CLMC}}_{L_{\max}} = \frac{1}{N}\sum_{k = 1}^{N}\sum_{j=1}^{J^{(k)}}\sum_{i=0}^{n_p-1}a_{ij}^{(k)}\sum_{m=0}^i (-1)^m \frac{i^m}{r^{m+1}}\left(\big(\tilde{\ell}_j^{(k)}\big)^{i-m}e^{r\tilde{\ell}_j^{(k)}} - \big(\ell_{j-1}^{(k)}\big)^{i-m}e^{r\ell_{j-1}^{(k)}}\right) \,, 
\end{equation}
where $\tilde{\ell}_j^{(k)}$ is defined as in \eqref{elltilde}. Note that, when the the regression polynomial is a simple piecewise linear interpolation polynomial, the CLMC estimator \eqref{CLMCest_pract2} reduces to \eqref{CLMCest_pract}.

\subsubsection{Quadrature and higher-order differences}

It is also possible to derive alternative practical methods from the
fundamental CLMC equation in \eqref{CLtel_sum}, by using alternative
approximations of the integral and the derivative. In order to
simplify the presentation, let us assume $L$ to be constant. 

Standard MLMC can be interpreted as an estimator for the
right hand side of \eqref{CLtel_sum} that uses a backward rectangular
quadrature rule on a uniform mesh\footnote{For any integrable
  function on $(0,L)$, this is defined as $\int_0^L f(\ell)\,d\ell \approx
  h\sum\limits_{i=0}^{M-1} f(ih)$, where $h =L/M$ and $M\in\N$.} with
the derivative approximated by a backward finite difference. 
This choice of quadrature rule and finite difference approximation is
special, because it is in fact exact for this simple case. However, in
general one could also
pick other schemes, perhaps exploiting more points and therefore
catching more global information, at the price of introducing a
correction term for both of the extremes of the interval $[0,L]$ that
will also need to be estimated (this
will be made clearer in the example below).
In particular, it is possible to come up with finite difference schemes which provide better variance reduction than the standard differences in MLMC.

Here, we just give a single example to make the basic idea
clearer. For sake of notation, we will denote the approximation terms
with the level as subscript rather than as argument.
 
MLMC exploits the following approximation of the derivative:
\begin{equation}\label{2derivapprox}
\frac{\td Q(\ell)}{\td\ell} \approx \frac{Q_\ell - Q_{\ell-h}}{h}\,, 
\end{equation}
for some $h>0$. Another possible derivative approximation scheme is given by the five-point stencil formula:
\begin{equation}\label{5stencil}
\frac{\td Q(\ell)}{\td\ell} \approx \frac{Q_{\ell-2h} -8 Q_{\ell-h} +8 Q_{\ell+h} -Q_{\ell+2h}}{12h}\,.
\end{equation}
Let us call
\[ v := \lim_{\ell\to\infty} \V[Q_\ell]\,,\quad\quad c := \lim_{\ell\to\infty} \Cov\big(Q_\ell,Q_{\ell+h}\big)\,.\]
Then, in the limit $\ell\to\infty$, with the derivative approximation in \eqref{2derivapprox} we have
\[ \V\left[\frac{\td Q(\ell)}{\td\ell}\right] \approx\V\left[\frac{Q_\ell - Q_{\ell-h}}{h}\right]\to \frac{2}{h^2}(v-c)\,, \]
whereas with the derivative approximation in \eqref{5stencil} we have
\[ \V\left[\frac{\td Q(\ell)}{\td\ell}\right] \approx\V\left[\frac{Q_{\ell-2h} -8 Q_{\ell-h} +8 Q_{\ell+h} -Q_{\ell+2h}}{12h}\right]\to \frac{130}{144h^2}(v-c)\,. \]
This shows that, for $\ell$ big enough, the five-point stencil formula
provides more than double the variance reduction with respect to the
scheme used by MLMC.

In general, it can be shown that since the coefficients of any finite difference derivative approximation have to sum up to 0, the variance of the related estimator can always be asymptotically written as some constant times $v-c$. This guarantees that, for any of these approximation schemes, the variance decreases to 0 as the covariance increases.

A practical formula for the five-point stencil CLMC method can be written as
\[ \Ev[Q(L)] = \Ev[Q_{2h}] + \frac{1}{12}\sum_{i=2}^{M-1}\Ev[Q_{(i-2)h} -8 Q_{(i-1)h} +8 Q_{(i+1)h} -Q_{(i+2)h}] + \Ev[\Delta_0] +\Ev[ \Delta_L]\,, \] 
where $h=L/M$, for some $M\in\N$, and $\Delta_0$ and $\Delta_L$ are the 
correction terms at Level $0$ and $L$, respectively. They can be written as
\begin{align*}
\Delta_0 &= \frac{1}{12}(-Q_0 + 7Q_h - 5Q_{2h} - Q_{3h})\,,\\
\Delta_L &= \frac{1}{12}(Q_{(M-2)h} - 7Q_{(M-1)h} + 5Q_{Mh} + Q_{(M+1)h})\,.
\end{align*}
Note that, by using again the asymptotic argument given before, we
have $\V[\Delta_L]\to \frac{76}{144}(v-c)$ for $L\to\infty$, which
guarantees variance reduction also for the correction term $\Delta_L$. The
correction term $\Delta_0$ consists only of coarse approximations and is
therefore cheap to compute even if many samples are needed. Note,
however, that it corresponds to a finite difference approximation of a 
derivative at $\ell =0$ and thus its variance is typically
significantly smaller than $\V[Q_0]$.

\section{Application to Adaptive Multilevel Monte Carlo}\label{sec:AMLMC}

The development of the continuous level framework was motivated by the
challenge of integrating sample-wise adaptive finite element solutions
within a hierarchical framework. For a given sample, there are
significant computational gains to be realised by using goal-oriented
(towards the quantity of interest) schemes, particularly when the
random field or quantity of interest is localised. The exciting
conceptual idea here is in contrast to other adaptive multilevel MC
methods \cite{eigel2016adaptive,babushkina2016adaptive} we do not use
the refinement steps or some pre-defined error tolerances as the levels, but instead use a continuous measure of error in the quantity of interest as our level. This naturally fits within our CLMC framework.

\subsection{Subsurface Flow Problem \& Constructing Pathwise Adaptive
  Solutions} \label{sec:model}

We consider a toy-model describing steady state, single
phase, incompressible flow in a permeable medium (e.g. rock), given by the linear, scalar elliptic partial differential equation
\begin{equation}\label{eqn:Darcy}
-\nabla \cdot k({\bf x}) \nabla u({\bf x}) = f({\bf x}) \quad \forall {\bf x} \in D \subset \mathbb R^d,
\end{equation}
subject to suitable boundary conditions. Physically $u({\bf x})$ is
the fluid pressure, $f({\bf x})$ the fluid source term and $k({\bf
  x})$ the scalar permeability field. In practical applications
(e.g. in oil reservoir simulation), the permeability field $k({\bf
  x})$ or the source term $f({\bf x})$ are not known everywhere, therefore
a typical approach is to model each as a random field. Let the sample
space be denoted by $\Omega$, then the random permeability and source
field  $k({\bf x}, \omega)$ and $f({\bf x}, \omega)$ belong to $D
\times \Omega$ with a certain distribution (inferred from
data). Therefore the solution to \eqref{eqn:Darcy}, the unknown
pressure field, is also a random field i.e. $u({\bf x}, \omega) \in D
\times \Omega$. For simplicity, we shall restrict ourselves to
homogeneous Dirichlet conditions $u(\omega,\cdot) \equiv 0$ on the
domain boundary $\partial D$.

For a fixed $\omega \in \Omega$ we can recast \eqref{eqn:Darcy} as a
standard variational problem, i.e. find $u({\bf x},\omega) \in V :=
H^1_0(D) = \{ v \in H^{1}(D) : v = 0 \; \mbox{on} \; \partial D \}$, such that 
\begin{equation}\label{eqn:variational_problem}
\underbrace{\int_D k({\bf x},\omega) \nabla u \cdot \nabla v \; d{\bf
    x}}_{=: \; a(\omega;u,v) } \; = \; \underbrace{\int_D
f({\bf x},\omega) v\;d{\bf x}}_{=: b(\omega;v)}\; \,, \qquad \forall v \in V\,.
\end{equation}
Here, $D$ is assumed to be  a bounded Lipschitz domain and $V =
H^1_0(D)$ is the usual Sobolev space of weakly differentiable
functions on $D$. Then, $a(\omega;\cdot,\cdot)$ is a symmetric, bounded and 
positive-definite bilinear form on $V \times V$, and as such
defines an inner product and a norm on $V$,\linebreak
the so-called energy norm $\|u\|_a := \sqrt{a(u,u)}$. If $f$ is
sufficiently smooth, then the 
functional $b(\omega; \cdot)$ is bounded on $V$.

To approximate the pressure solution $u({\bf x},\omega)$, we construct
a (sample-wise adapted) finite element (FE) space $V_h(\omega) \subset V$ of piecewise
linear Lagrange polynomials on a grid $\mathcal T_h(\omega)$ that
vanish on the boundary of $D$. The FE solution $u_h({\bf x},\omega) \in V_h(\omega)$ satisfies 
\begin{equation}\label{eqn:variational_darcy}
a(\omega;u_h,v_h) = b(\omega;v_h)\,, \qquad \forall v_h \in V_h(\omega),
\end{equation}
resulting in a (large) linear system of equations of dimension $M_h(\omega) :=
\dim(V_h(\omega))$. From this, we are interested in approximating
statistics (e.g. the expected value) of a {\em quantity of interest}~$\cQ$, defined to be (for simplicity) a linear functional of $u_h({\bf x},\omega)$.

As motivated at the beginning of this section, we are going to build
our approximate solutions, sample-by-sample using adaptive finite
element methods. But instead of using the number of
refinement steps as the level parameter and applying MLMC, we will use a sample-wise error estimate as
the level parameter and apply our new CLMC framework.  

For any $\omega \in \Omega$, starting with an initial grid $\mathcal
T^{(0)} (\omega)$, chosen to be the same for each sample, we use an
$h$-adaptive refinement strategy to construct a sequence of grids
$\mathcal T^{(k)} (\omega)$ for $k = 0,\ldots,K$. In our case, the
adaptive procedure is driven by a local, {\em goal-orientated} error
indicator $e_\tau^{(k)}(\omega)$, for each $\tau \in
\mathcal T^{(k)}(\omega)$. This gives the relative contribution from
each element to the error in the quantity of interest $\cQ(u(\omega))$, so that 
\begin{equation}\label{eqn:error_estimator}
|\cQ(u(\omega)) - \cQ(u^{(k)}(\omega))| \le e^{(k)}(\omega) = \left( \sum_{\tau \in \mathcal T^{(k)}(\omega)} e_\tau^{(k)}(\omega) \right)^{1/2}\,.
\end{equation}

In addition to solving 
\eqref{eqn:variational_darcy} (the so-called {\em primal problem}), goal-oriented error
estimators typically also require an approximate FE solution
$w_h(\omega,{\bf x})$ of the {\em dual problem}
\begin{equation}\label{eqn:dual}
a(\omega;v_h,w_h) = \cQ(v_h) \quad \forall v_h \in V_h.
\end{equation}
 There are many different choices of goal-oriented error
estimators, see for example \cite{Gra05}.
For one particular choice,
described in detail in \cite{Gra05}, the error estimator $e_\tau^{(k)}(\omega)$ in each
element $\tau \in \mathcal T^{(k)}$ is computed by bounding the product of the
energy norms of the errors in the primal and dual FE solutions
$u_h(\omega,{\bf x})$ and $w_h(\omega,{\bf x})$ of
\eqref{eqn:variational_darcy} and  \eqref{eqn:dual}, respectively. 
Up to a sample-dependent
constant, these bounds are simply the sum of the element residuals and
of the jumps/discontinuities in inter-element fluxes for each of the two
problems. Full details can be found in \cite{Gra05}, but we will also
provide some more details in Appendix
\ref{sec:error_estimator}. 

The FE grid $\mathcal T^{(k+1)}(\omega)$ is generated by
refining the $\theta^{(k)}$ percent of elements of $\mathcal T^{(k)}(\omega)$
that contribute most to the error in $\cQ$ as defined by 
\eqref{eqn:error_estimator}. This is typically followed by some
additional refinements that ensure that the FE space $V^{(k+1)}(\omega)$ is conforming,
i.e. that there are  no hanging nodes in $\mathcal T^{(k+1)}(\omega)$.
In our numerical experiments below, we increase $\theta^{(k)}$ as $k$ increases
and use a so-called {\em red/green refinement} strategy that ensures conformity.

Finally, we now define our sample-wise {\em continuous level}
at refinement step $k$ to be
\begin{equation}\label{eqn:error_measure}
\ell_k(\omega) = - \log\left( \frac{e^{(k)}(\omega)}{e^{(0)}(\omega)}\right)
\end{equation}
The level gives a sample-wise measure of the error in
$Q_k(\omega)$, the quantity of interest computed on $\mathcal
T^{(k+1)}(\omega)$, relative to the error on the coarsest grid. We note
that with this choice, computations on $\mathcal T^{(0)}$ are naturally
providing values $Q_0(\omega)$ at level
$\ell_0(\omega) = 0$. However, the main reason for defining the error
in this way is due to the explicit error estimator that are being used
being only known up to an unknown constant (dependent on $\omega$).

\subsection{Numerical Experiments}

All the numerical experiments are calculated using the high
performance FE library DUNE~\cite{Bas10} and its discretisation module
\texttt{dune-pdelab}. Simulations are carried out on a computer
consisting of four, 8-core Intel Xeon E5-4627v2 Ivybridge processors,
each running at 1.2 GHz, giving a total of 32 available cores. The
solutions for each sample are computed on a single processor and
independent samples are equally distributed across all available
cores. Individual solutions of the forward and dual problems are obtained using
the sparse direct solver \texttt{UMFPACK} \cite{Dav04}. Each adaptive step uses the
{\em red/green refinement} strategy, as implemented in
\texttt{dune-grid} \cite{Bas08}, refining $\theta^{(k)}$ percent of
elements from $\mathcal{T}^{(k)}$ to $\mathcal{T}^{(k+1)}$.

In our numerical test, we consider $D := [0,1]^2$. The coarse grid $\mathcal T^{(0)}$ for all samples is taken as
a uniform $32 \times 32$ triangular mesh on $D$. In our test we
consider \eqref{eqn:Darcy} with random permeability field $k$
and random source term $f$. The permeability field $k({\bf
  x},\omega)$ is characterised by a log-normal random field, where
$\log k({\bf x},\omega)$ has a mean of zero and a two-point
exponential covariance function
\begin{equation}\label{eqn:covariance_function}
C({\bf x},{\bf y}) := \exp\left( -3\,\|{\bf x}-{\bf y} \|_1\right) \quad {\bf x},{\bf y} \in D,
\end{equation}
with $\|\cdot\|_p$ denoting the $\ell_p$-norm in $\mathbb R^2$. The field is parameterised with a (truncated) Karhunen-Lo\`eve (KL) expansion
 \begin{equation}\label{eqn:KL}
 k({\bf x},\omega) = \exp\left(\sum_{i=1}^R \sqrt{\mu_i}\phi_{i}({\bf x}) \xi_i\right).
 \end{equation}
where $\{\mu_i \}_{i\in\mathbb N}$ are the eigenvalues, $\{\phi_i({\bf
  x}) \}_{i\in\mathbb N}$ the corresponding $L_2$-normalised
eigenfunctions of the covariance operator with kernel function $C({\bf
  x},{\bf y})$ and $\xi_i \sim \mathcal N(0,1)$. For more details on
how this expansion is constructed see for example
\cite{cliffe2011multilevel}. In the calculations which follow we take $R=36$. For the random source term, we take
\begin{equation}
f({\bf x},\omega) = 1000 \,a\; \exp \left( -20 \|{\bf x} - {\bf y}_f \|^2_2\right)
\end{equation}
where $a$ and the components of ${\bf y}_f$ are all sampled from $\mathcal U(0,1)$.

As the quantity of interest, we consider the average pressure near ${\bf y}_Q := [0.25, 0.25]^T$, defined by the linear functional
\begin{equation}\label{eqn:Test2_QoI}
\cQ(u) := C_1\int_{D} \exp\left(-\frac{\| {\bf x} - {\bf y}_Q\|_{2}^{2}}{\lambda_Q}\right) u({\bf x},\omega) d{\bf x},
\end{equation}
with $\lambda_Q = 0.0005$ and $C_1 =
\left(\int_{D} \exp(-\| {\bf x} - {\bf y}_Q\|_{2}^{2}/\lambda_Q) d{\bf
    x}\right)^{-1} \approx 0.00157$. 

We now test our CLMC algorithm (Algorithm \ref{alg_CLMC}) by comparing
uniform refinements and adaptive refinements with a variable $\theta^{(k)}$
(percentage of elements refined per step). In particular, we choose
\begin{equation}
\theta^{(k)} = \min(100\%,\delta^k\theta_0)
\end{equation} 
as the percentage of elements refined in $\mathcal{T}^{(k)}$, with 
$\theta_0 = 1\%$ and $\delta = 3$. We note that this choice
is heuristic, motivated by a series of test runs. For the problem
at hand, the idea of starting with small $\theta^{(0)}$ and increasing
the percentage with the number of adaptive steps makes sense. Initially the error in
$\cQ$ is dominate by the fact that the grid is not well adapted to the
particular random sample $\omega \in \Omega$. This includes the random
field, the location of the localised source and the quantity of
interest itself. Once the adaptive strategy has focused in on all
those localised regions, the error in $\cQ$ is governed by the global lack of
singularity in the coefficient
\cite{charrier2013finite,teckentrup2013further} and thus distributed fairly
uniformly across the whole domain. So from that point onwards, refining all
elements uniformly leads to the most effective error reduction.

\begin{figure}[t]
\centering
\includegraphics[width = \linewidth]{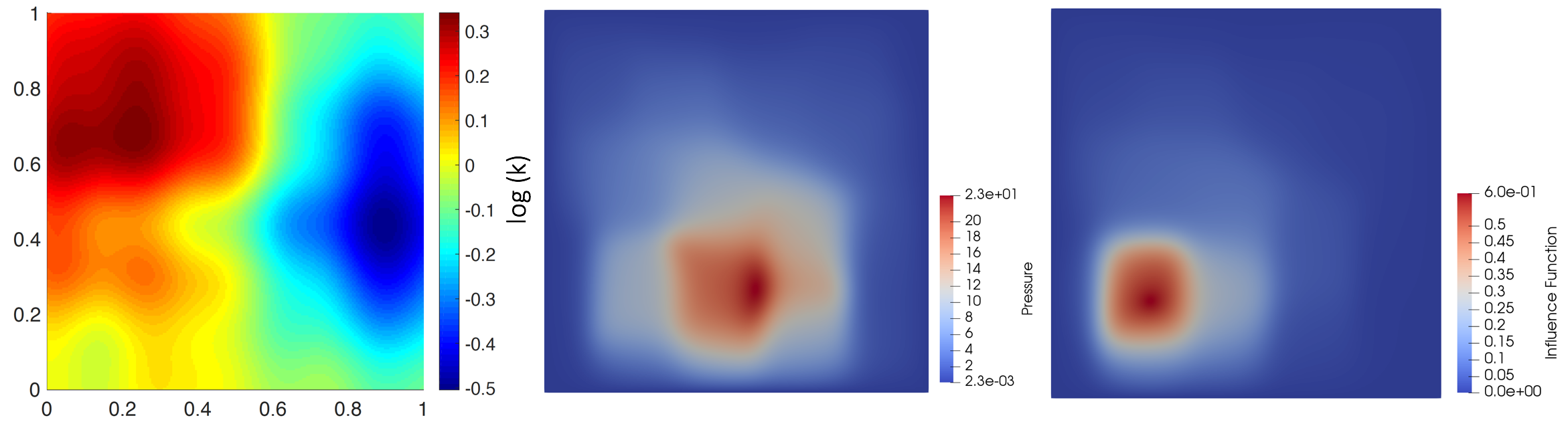}
\caption{Permeability field $k$, pressure 
solution $u_h$ and influence function $w_h$ on the finest adaptive grid ($k=6$) for a particular realisation 
$\omega \in \Omega$.\label{fig:solution_TC1}}
\end{figure}

Before running a complete simulation we first consider a single sample
$\omega \in \Omega$. Figure~\ref{fig:solution_TC1} shows the random
permeability field $k({\bf x},\omega)$, pressure solution $u_h({\bf x},\omega)$, and the influence function 
$w_h({\bf x},\omega)$ (i.e. the solution of the dual problem \eqref{eqn:dual})
for this sample after $6$ adaptive
steps. Snapshots of the grids, built
using the goal-oriented error estimator, are shown in
Figure~\ref{fig:example_grids_TC1} at steps $0$, $2$, $4$ and
$6$. Visually, we see that the adaptive
scheme is working correctly, refining near ${\bf y}_Q = [0.25,0.25]^T$,
the point around which the pressure is averaged in the functional
$\cQ$ in \eqref{eqn:Test2_QoI}, whilst also adapting around the
localised source. At the latter levels the refinement also starts to
pick up local variations in the
permeability field in regions that influence the
pressure at the point of  interest.
\begin{figure}[t]
\centering
\includegraphics[width = 1.0\linewidth]{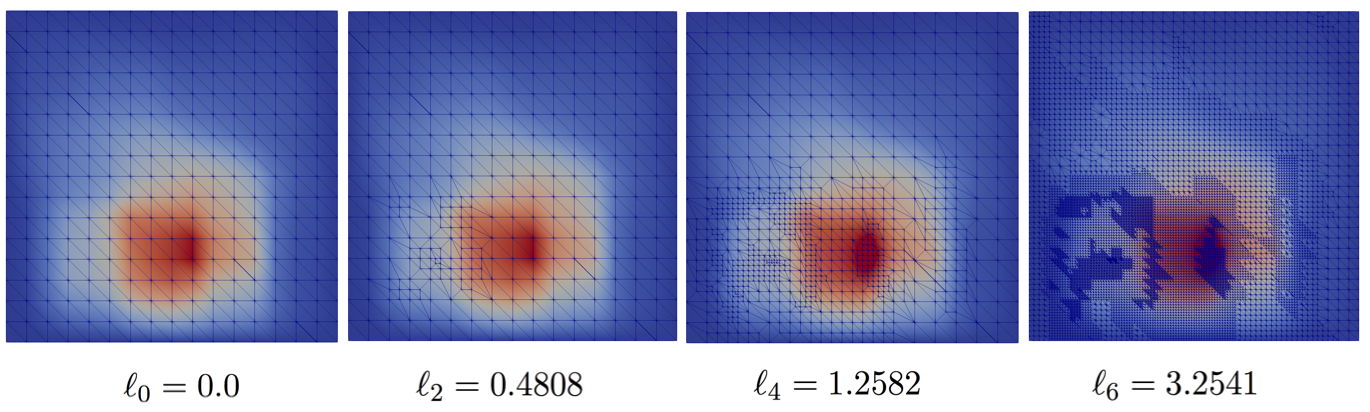}
\caption{Sequence of adaptive grids built using goal-oriented error estimator for random $\omega \in \Omega$, level is defined by $\ell_k$ given by \eqref{eqn:error_measure}.\label{fig:example_grids_TC1}}
\end{figure}


For the uniform and adaptive strategy, we first run an initial batch
of $6400$ samples up to $L_{\max} = 5$, in order to estimate the
parameters $\beta$ and $\gamma$. In a real simulation, it would not be
necessary to estimate these parameters accurately and so significantly
fewer samples could be used. With uniform refinements, our
estimates are $\beta_u=2.28$ and $\gamma_u=1.0$, whereas for adaptive
refinements we get $\beta_a=2.22$ and $\gamma_a=0.78$. Note that, in both
cases, $\beta>\gamma$, therefore by taking  $L_{\max}=+\infty$ in the
CLMC setting we obtain unbiased estimators with respect to
$\Ev[\cQ-Q(0)]$. In these initial runs we can already see the expected
computational gains of adaptive grid refinement. We note that the
rates $\beta$ for $\mathbb V[dQ/d\ell]$ are much the same in each
case, whilst $\gamma$, the rate of growth of the expected cost per
sample, is clearly smaller for the adaptive strategy. Figure
\ref{fig:costperSample} gives a plot of the continuous level $\ell$,
representing the estimate of the relative finite element error,
against the natural log of the cost for all samples, which shows the
better rate for the adaptive scheme.

\begin{figure}
\centering
\includegraphics[width=0.4\linewidth]{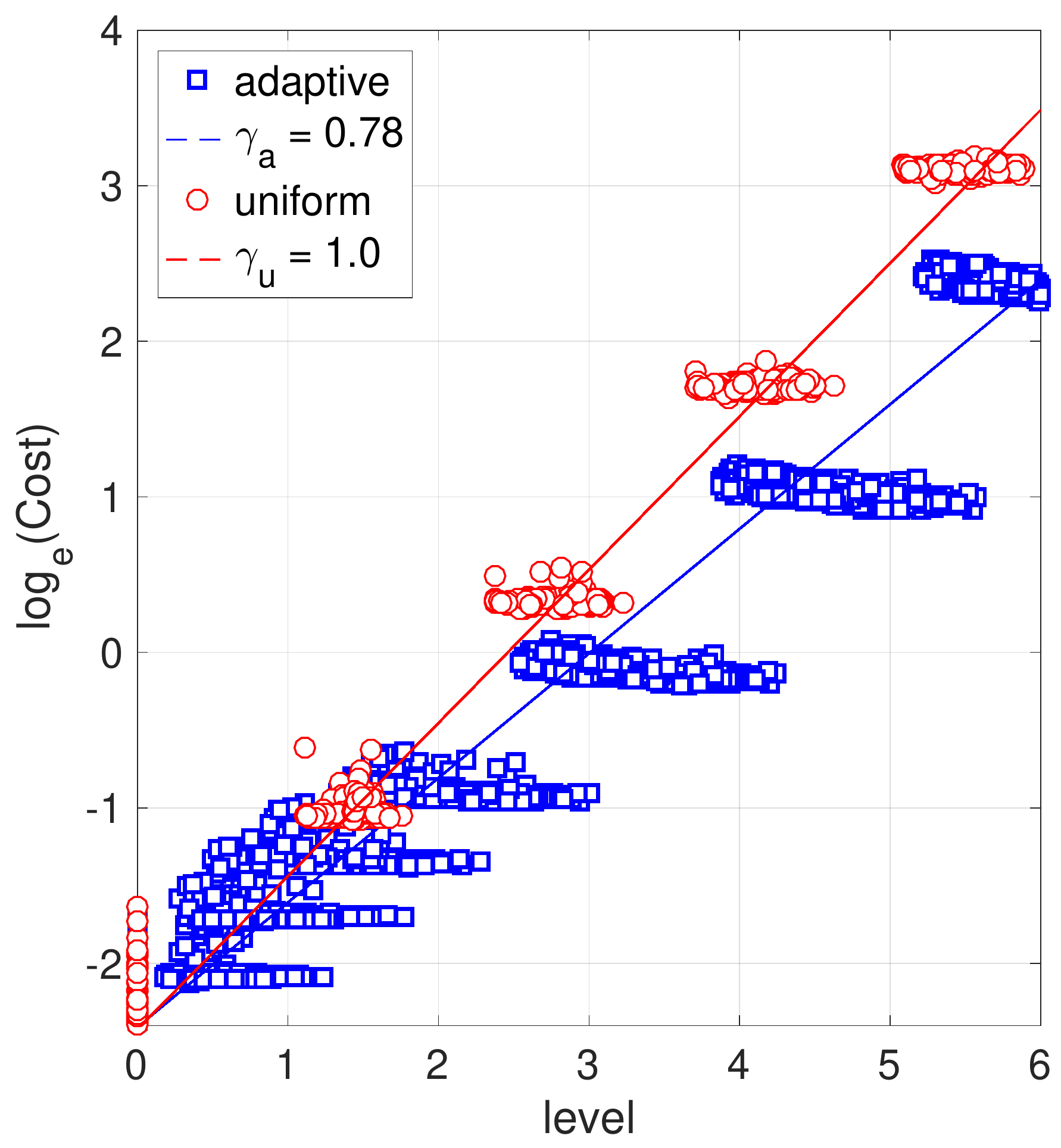}
\caption{Level $\ell$ against $\log($Cost$)$ for 6400 uniform (red
  circles) and 6400 adaptive samples (blue squares).\label{fig:costperSample}}
\end{figure}

\begin{figure}[t]
\centering
\includegraphics[width = 0.30\linewidth]{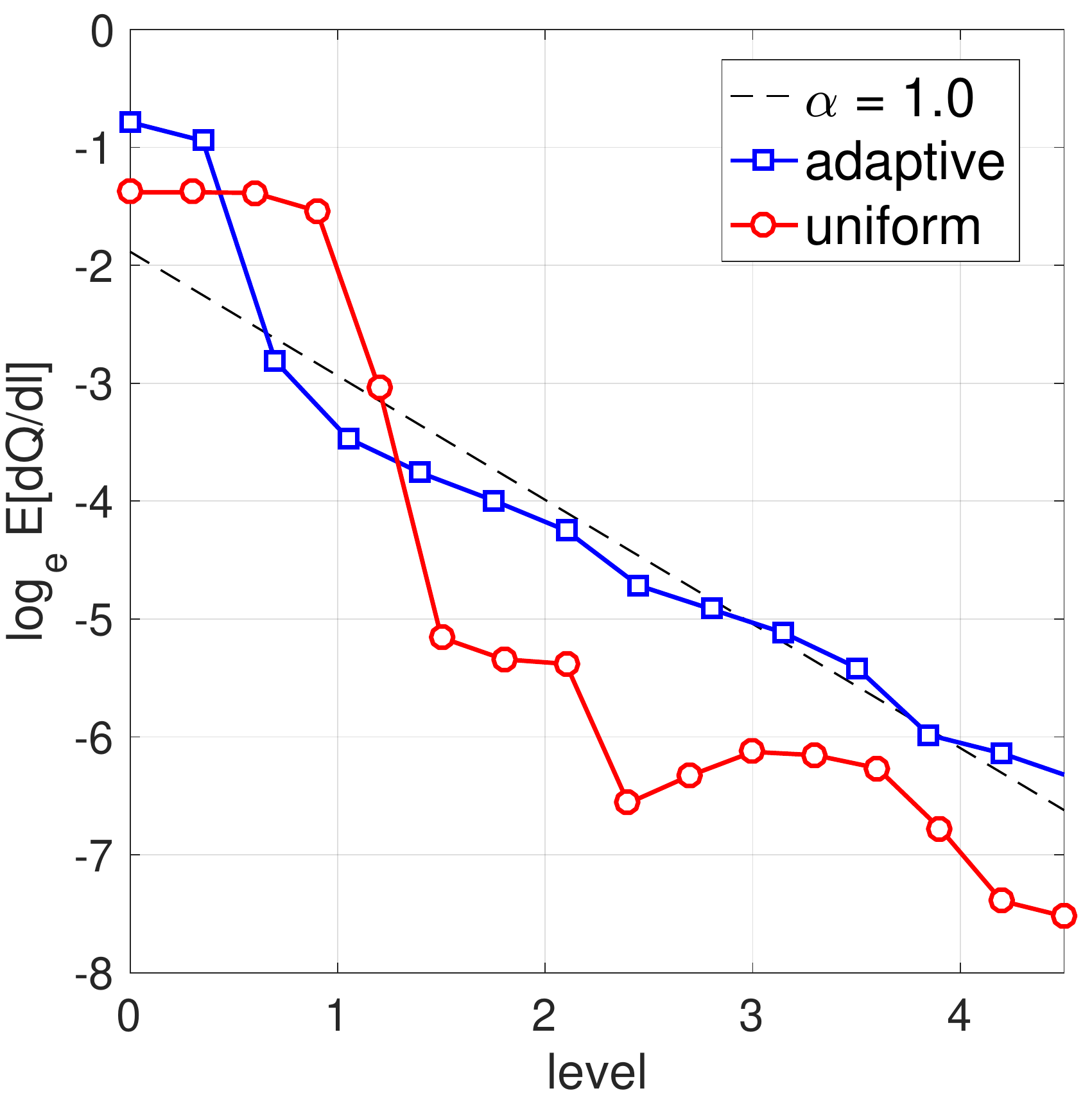}\includegraphics[width = 0.30\linewidth]{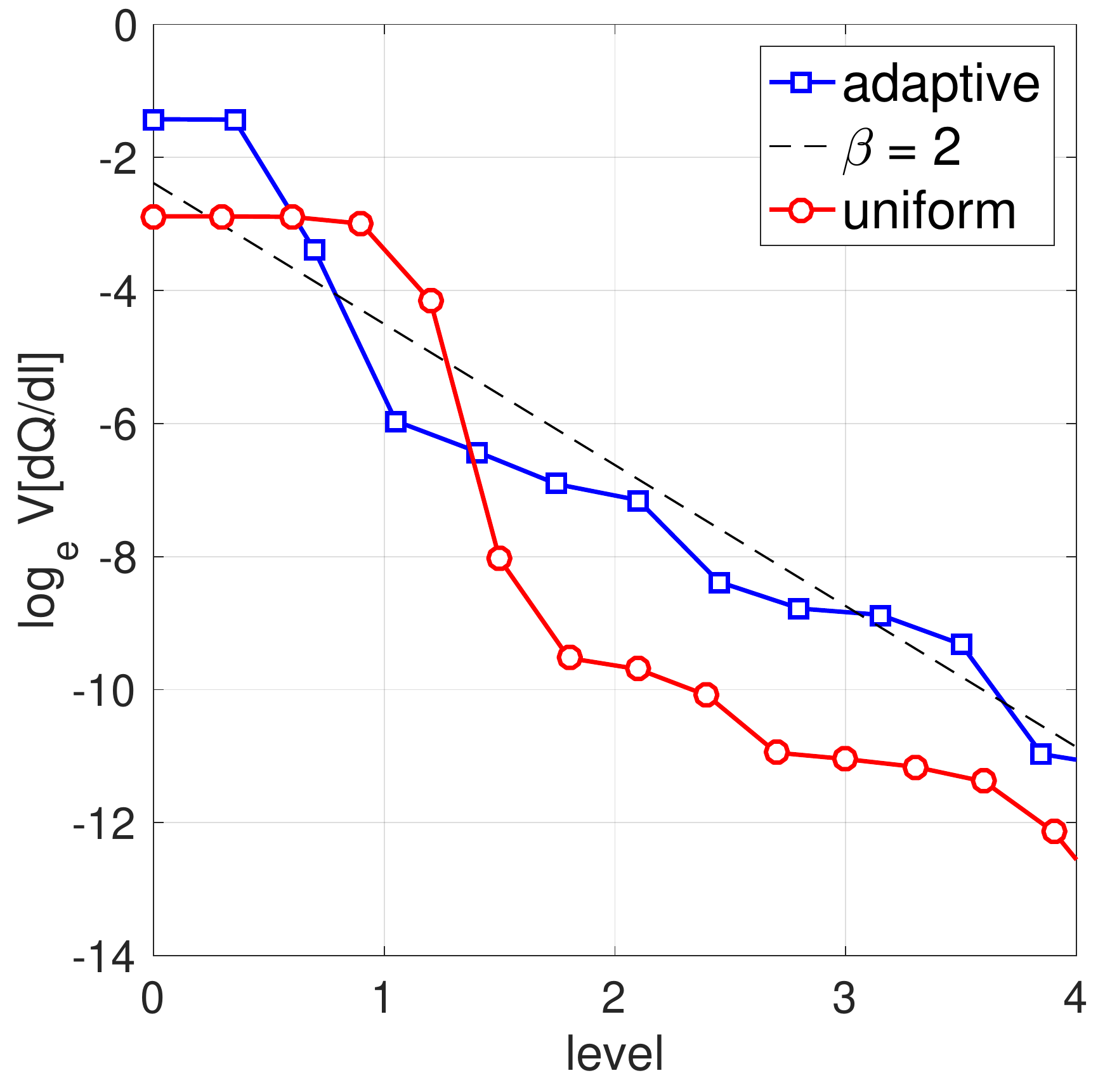}\includegraphics[width = 0.32\linewidth]{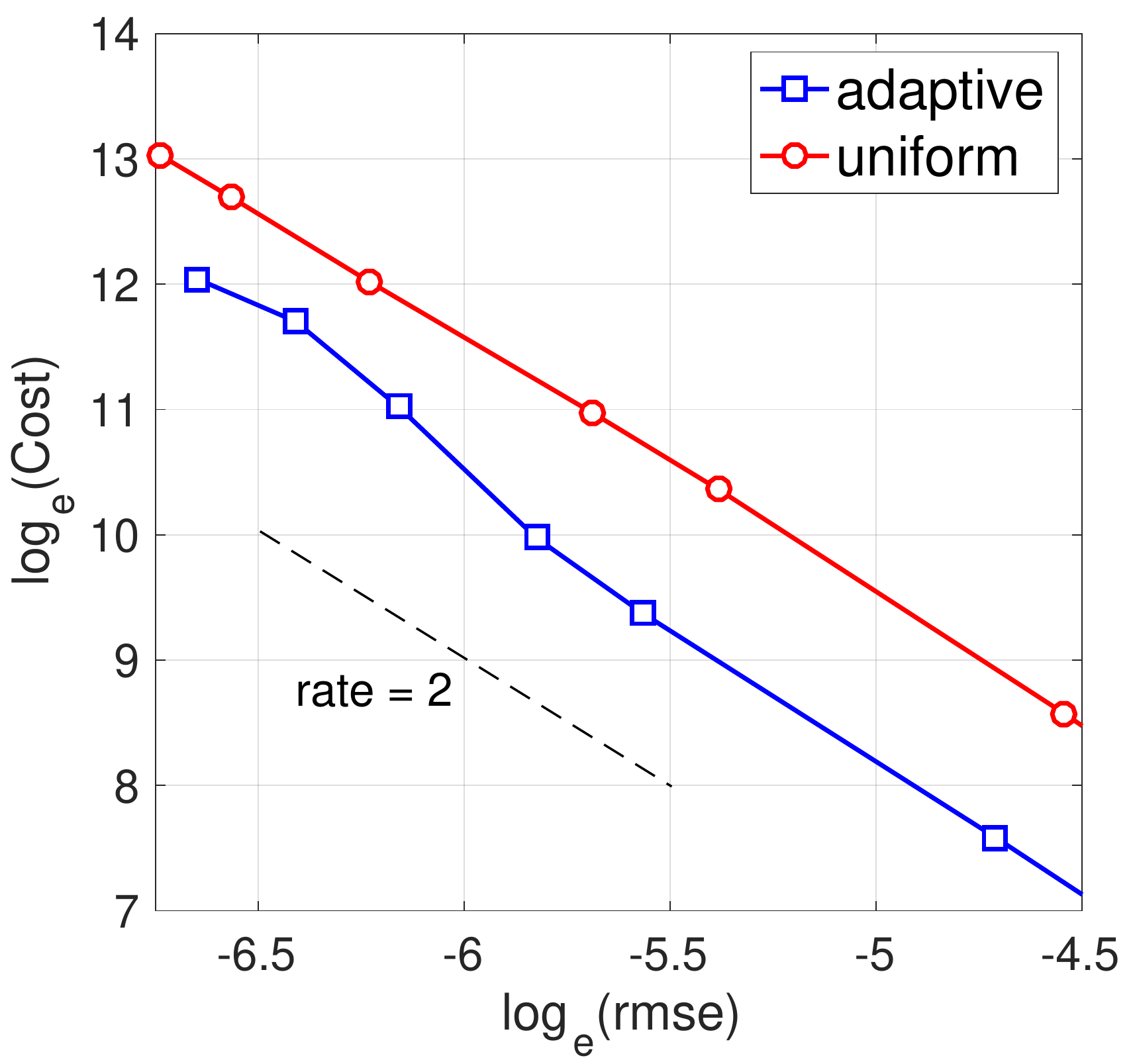}
\caption{Results for the numerical test, in log-scale. Left and middle:
  Convergence plots of $\mathbb E[dQ/d\ell]$ and $\mathbb V[dQ/d\ell]$
  against $\ell$ respectively. Right: Total cost of uniform and
  adaptive algorithm (in seconds) against estimated sampling error
(= root mean square error due to unbiasedness).\label{fig:results_TC1}}
\end{figure}
We then run the CLMC algorithm with a maximum of $N=10^6$ samples for each
case. The exponential parameter rate $r$ is taken to be the same for
each case, so that any computational gains can be attributed to the
adaptive strategy, rather than a difference in $r$. The value is
chosen so that $r = \frac{1}{2}(r_u + r_a) = \frac{1}{4}(\beta_u +
\gamma_u + \beta_a + \gamma_a) = 1.57$, and we consider the unbiased
estimator with $L_{\max} = +\infty$.

The numerical results show that the CLMC algorithm is working as
expected. In Figure~\ref{fig:results_TC1} (left), we observe
as expected that the natural logarithm of $\mathbb E[dQ/d\ell]$
decreases linearly with $\ell$, i.e. $\alpha \approx 1$, in both the
uniform and the adaptive case, since $\ell$ is defined as the natural
logarithm of an estimate of the relative bias
error. Figure~\ref{fig:results_TC1} (middle) shows the variance
reduction for both uniform and adaptive refinement strategies. Both
decay very similarly across the levels with rates of around $\beta = 2$.
Finally, Figure~\ref{fig:results_TC1} (right) shows the actual cost to
compute the estimate for different choices of $N$. The cost (in
seconds) is plotted against the root mean square error, which is equal
to the sampling error, since the estimator is unbiased. As proved in
Theorem~\ref{thm_CLMC_giles}, since $\beta>\gamma$ for both
strategies, we observe parallel straight lines with rate of
$\approx 2$. Due to the reduced computational cost on the finer
levels, the adaptive strategy wins over the uniform one across a range of
tolerances. Especially for coarser tolerances the gains are
significant and the sample-adaptive level hierarchy consistently
reduces the cost by a factor of $4$.

The actual gains that are possible with the new CLMC estimator and
with sample-adaptive level hierarchies are very problem
dependent. They also depend strongly on the error estimator and on the
adaptive refinement strategy. The estimator and the strategy employed
here are by no means optimal. It is known that the employed error
estimator is not necessarily very effective in the context of strong
coefficient variations. Finally, the gains also depend on the cost of
the linear solver. For a fair comparison, we used a sparse direct
solver, which outperforms iterative solvers for the problem sizes
encountered in our 2D model problem. However, further experiments in
three space dimensions will require iterative solvers and robust
preconditioners that can cope both with the strong coefficient
variations and with the locally refined finite element
meshes. The cost and the memory requirements of sparse direct solvers
grow too rapidly in 3D. Nevertheless, we expect the gains in 3D to be even more
significant.

\section{Conclusions \& Further Work}
\label{sec:conclusion}

In this paper, we introduce Continuous Level Monte Carlo (CLMC), a
generalisation of MLMC to a continuous framework where the level is a
continuous variable rather than an integer. We propose a
practical estimator and prove a Complexity Theorem, showing the
same order of convergence as in MLMC. Furthermore, we provide a
version of the estimator that is unbiased with respect to the true
quantity of interest and extend the Complexity Theorem to this
case, giving sufficient and necessary conditions for the unbiased
estimator to have finite cost. We apply CLMC to adaptive refinement
schemes, where the continuous framework is particularly well suited in
order to capture sample-based level hierarchies. We demonstrate
clear computational gains when adaptive refinement strategies are adopted
rather than uniform ones. 

The introduction of CLMC opens the door 
to several new research directions. We outline a few ideas for further
work:\smallskip

\noindent {\em Extension of Multi-Index Monte Carlo (MIMC)}
  \cite{hajimulti}. MIMC is an extension of MLMC to
multi-dimensional level parameters and higher-order differences. 
In the same way, as CLMC generalises MLMC 
by replacing the sum with an integral and the difference with a
derivative in the case of a scalar level parameter, one could
generalise MIMC by employing multi-dimensional
integrals of partial derivatives. Indeed, consider
$(Q(\bs\ell))_{\bs\ell}$ to be a sequence of approximation functions
of $\cQ$, where $\bs\ell=(\ell_1,\dots\ell_m)$ is a $m$-dimensional
vector of non-negative levels. To explain the idea, let us restrict our
description to $m=2$ and consider a $2$-dimensional positive
random variable $\bs L=(L_1,L_2)$. Assuming sufficient
regularity, we can write
\begin{equation}\label{eq:CIMC}
\Ev\big[Q(\bs L) - Q(\bs 0)\big] =
\Ev\left[\int_0^{L_1}\!\!\!\!\int_0^{L_2}\frac{\de^2Q(\bs\ell)}{\de
    \ell_1\de\ell_2}\,d\bs\ell \right]  \ + \ \sum_{j=1}^2 \;
\Ev\left[\int_0^{L_j} \frac{\de Q(\bs\ell)}{\de
    \ell_j}\,d\ell_j \right] \,.
\end{equation}
Note that \eqref{eq:CIMC} is a two-dimensional extension of the formula in
\eqref{CLtel_sum}. It is outside the scope of this paper, but we argue
that different choices for the probability
distribution of the vector of finest levels $\bs L$ (with potentially
correlated components) correspond to different choices of the grid of
levels in MIMC. A natural choice would be again to pick independent
$L_i\sim\text{Exponential}(r_i)$, for $i=1,\dots,m$, with $r_i>0$. 
Classically, in MIMC, $\bs L$ is a fixed integer vector chosen to control the
bias error, while the optimal strategy for the choice of samples
avoids computation of samples for levels with $\ell_1/L_1 + \ell_2/L_2
> 1$. Here, the bias can again be completely eliminated (provided the
variance decays fast enough w.r.t. the growth in cost), 
and the optimal strategy is a direct consequence of the choice of the
exponential distributions for $L_1$ and $L_2$, making the probability that both
$\ell_1$ and $\ell_2$ are simultaneously large practically zero.\smallskip

\noindent {\em Extension of Multilevel Monte Carlo Markov Chain
  (MLMCMC) \cite{dodwell2015hierarchical}.} Multilevel techniques have
been successfully applied to sampling algorithms like MCMC,
drastically reducing their complexity cost. The extension of MLMCMC to
Continuous Level MCMC is object of future work, potentially leading to
an estimator that is unbiased with respect
to the real quantity of interest, under the real target probability
distribution. Such an unbiased estimator would be of great interest:
unlike forward problems, where the bias can arise only from the
approximation of the quantity of interest, inverse problems have the
additional issue of an approximation of the target probability
distribution. Unbiasedness guarantees that the estimator is in fact
estimating the correct unknown, without expensive extra computational
cost to estimate the bias error. In addition, continuous level
adaptive refinement strategies will significantly help to slim down MCMC's
computational cost, allowing to solve even more complex problems.\smallskip
  

\appendix

\section{Proof of the Complexity results}\label{app_CT}

\subsection{Proof of Theorem \ref{thm_CLMC_giles}}
\label{proof_thm}
\begin{proof}
First, we want to bound the MSE by $\varepsilon^2$. By the
bias-variance decomposition, this can be achieved by bounding both the
squared bias and variance by $\varepsilon^2/2$.

By using assumption \textit{(i)} and recalling that $L\sim\textnormal{Exponential}(L)$, the bias term is bounded by
\begin{align}
\Bigg|\Ev\left[\widehat{Q}^{\textnormal{CLMC}}_{L_{\max}} - (\cQ-Q(0))\right]\Bigg| &= \Bigg|\Ev\left[\int_{L\wedge L_{\max}}^L\frac{1}{\Py(L\ge \ell)}\frac{\td Q(\ell)}{\td \ell}\,d\ell\right]\Bigg| \notag\\
&\le \Ev\left[\int_{L\wedge L_{\max}}^L\frac{1}{\Py(L\ge \ell)}\Bigg|\Ev\left[\frac{\td Q(\ell)}{\td \ell}\right]\Bigg|\,d\ell\right] \notag\\
&\le c_1 \Ev\left[\int_{L\wedge L_{\max}}^L\frac{1}{\Py(L\ge \ell)}e^{-\alpha\ell}\,d\ell\right] \notag\\[1ex]
&= \begin{cases}\frac{c_1}{r-\alpha}\Ev\left[e^{(r-\alpha)L}-e^{(r-\alpha)L\wedge L_{\max}}\right] &\mbox{if } r\ne \alpha\\[0.5ex]
c_1\Ev[L-L\wedge L_{\max}] &\mbox{if } r=\alpha
\end{cases}\label{biasbound_intermed}\\[1ex]
&= \frac{c_1}{\alpha}e^{-\alpha L_{\max}} \,, \label{biasbound}
\end{align}
where we can explicitly compute the expected values in
\eqref{biasbound_intermed} using the distribution of $L$.
 
As we want to bound the squared bias by $\varepsilon^2/2$, this is equivalent to bounding the bias by $\varepsilon/\sqrt{2}$, which can be achieved by setting
\begin{equation}
\label{boundLmax}
L_{\max} \ge \left\lceil \frac{1}{\alpha}\log\frac{\sqrt{2}c_1r\varepsilon^{-1}}{\alpha}\right\rceil\,.
\end{equation}

Then, let us provide an upper bound for the variance of the CLMC estimator \eqref{CLest}. By the law of total variance, we have
\begin{equation}\label{totvar}
\V[\widehat{Q}^{\textnormal{CLMC}}_{L_{\max}}] = \Ev\left[\V[\widehat{Q}^{\textnormal{CLMC}}_{L_{\max}}|L]\right] + \V\left[\Ev[\widehat{Q}^{\textnormal{CLMC}}_{L_{\max}}|L]\right]\,.
\end{equation}
Let us start by bounding the first term on the right-hand-side of \eqref{totvar}. We will use Cauchy-Schwarz inequality on the covariance, followed by assumption \textit{(ii)}. We have
\begin{align*}
\Ev\left[\V[\widehat{Q}^{\textnormal{CLMC}}_{L_{\max}}|L]\right] &= \Ev\left[\Cov\left(\widehat{Q}^{\textnormal{CLMC}}_{L_{\max}},\widehat{Q}^{\textnormal{CLMC}}_{L_{\max}}\,\big|\,L\right)\right]\\
&=\frac{1}{N}\Ev\left[\int_{[0,L\wedge L_{\max}]^2} \frac{1}{\Py(L\ge
  \ell)}\frac{1}{\Py(L\ge \ell')}\Cov\left(\frac{\td Q(\ell)}{\td
  \ell},\frac{\td Q(\ell')}{\td
  \ell'}\right)\,\td\ell\td\ell'\right]\\
&=\frac{1}{N}\Ev\left[\int_{[0,L\wedge L_{\max}]^2} \frac{1}{\Py(L\ge \ell)}\frac{1}{\Py(L\ge \ell')}\V\left[\frac{\td Q(\ell)}{\td \ell}\right]^{\frac{1}{2}}\V\left[\frac{\td Q(\ell')}{\td \ell'}\right]^{\frac{1}{2}}\,\td\ell\td\ell'\right] \\
&=\frac{1}{N}\Ev\left[\left(\int_0^{L\wedge L_{\max}} \frac{1}{\Py(L\ge \ell)}\V\left[\frac{\td Q(\ell)}{\td \ell}\right]^{\frac{1}{2}}\,\td\ell\right)^2\right] \\
&\le \frac{1}{N}c_2^2\Ev\left[\left(\int_0^{L\wedge L_{\max}} \frac{1}{\Py(L\ge \ell)}e^{-\frac{\beta}{2}\ell}\,\td\ell\right)^2\right]\\[1ex]
&= \begin{cases}\frac{1}{N}\frac{4c_2^2}{(2r-\beta)^2}\Ev\left[\big(e^{(r-\frac{\beta}{2})L\wedge L_{\max}}-1\big)^2\right] &\mbox{if } r\ne \beta/2\\[1ex]
\frac{1}{N}c_2^2\Ev[(L\wedge L_{\max})^2] &\mbox{if } r = \beta/2
\end{cases}\\[0.5ex]
&\le \begin{cases}\frac{1}{N}\frac{4c_2^2}{(r-\beta)(2r-\beta)^2}\left((2r-\beta) e^{(r-\beta)L_{\max}}-\beta\right) &\mbox{if } r\ne \beta/2,\beta\\[1ex]
\frac{1}{N}\frac{4c_2^2}{\beta^2}(\beta L_{\max} + 1) &\mbox{if } r= \beta\\[1ex]
\frac{1}{N}\frac{8c_2^2}{\beta^2} &\mbox{if } r = \beta/2\,.
\end{cases}
\end{align*}
On the other hand, the second term on the right-hand-side of \eqref{totvar} can be bounded as
\begin{align*}
\V\left[\Ev[\widehat{Q}^{\textnormal{CLMC}}_{L_{\max}}|L]\right] &= \frac{1}{N}\V\left[\int_0^{L\wedge L_{\max}} \frac{1}{\Py(L\ge \ell)}\Ev\left[\frac{\td Q(\ell)}{\td \ell}\right]\,\td\ell\right]\\
&\le \frac{1}{N}c_1^2\V\left[\int_0^{L\wedge L_{\max}} \frac{1}{\Py(L\ge \ell)}e^{-\alpha\ell}\,\td\ell\right]\\[1ex]
&= \begin{cases}
\frac{1}{N}\frac{c_1^2}{(r-\alpha)^2}\V\left[e^{(r-\alpha)L\wedge L_{\max}}-1\right] &\mbox{if } r\ne\alpha\\[0.5ex] 
\frac{1}{N}c_1^2\V\left[L\wedge L_{\max}\right] &\mbox{if } r=\alpha
\end{cases}\\[0.5ex]
&\le \begin{cases}
\frac{1}{N}\frac{c_1^2}{(r-\alpha)^2}\Ev\left[e^{2(r-\alpha)L\wedge L_{\max}}\right] &\mbox{if } r\ne\alpha\\[0.5ex] 
\frac{1}{N}c_1^2\V\left[L\right] &\mbox{if } r=\alpha
\end{cases}\\[0.5ex]
&= \begin{cases}
\frac{1}{N}\frac{c_1^2}{(r-2\alpha)(r-\alpha)^2}\left(2(r-\alpha)e^{(r-2\alpha)L_{\max}}-r\right) &\mbox{if } r\ne\alpha,2\alpha\\[1ex]
\frac{1}{N}\frac{2c_1^2}{\alpha}L_{\max} &\mbox{if } r=2\alpha\\[1ex]
\frac{1}{N}\frac{c_1^2}{\alpha^2} &\mbox{if } r=\alpha\,.
\end{cases}
\end{align*}
In both cases in the last step, we have again used our knowledge of the distribution of
$L$. 

Note that asymptotically the bound for the first
term on the right-hand-side of \eqref{totvar} always dominates the
bound of the second, since 
we have assumed that 
$\beta \le 2\alpha$. Hence, adding together 
the two bounds and using \eqref{boundLmax}, as well as the fact that
$\varepsilon < e^{-1}$, we obtain the following 
asymptotic bound on the total variance:
\[
\V[\widehat{Q}^{\textnormal{CLMC}}_{L_{\max}}] \; \le \;
\frac{C'}{N} \begin{cases}
\varepsilon^{\frac{\beta-r}{\alpha}} 
&\mbox{if } r > \beta\\
\log \varepsilon &\mbox{if } r=\beta\\
1 &\mbox{if } r < \beta\,,
\end{cases}
\]
for some constant $C'>0$ that is independent of $N$ and
$\varepsilon$. Thus, to guarantee 
$\V[\widehat{Q}^{\textnormal{CLMC}}_{L_{\max}}] \le \varepsilon^2/2$
it suffices to choose
\begin{equation}\label{N_proof}
N \ge 2C' \varepsilon^{-2-\max(0,\frac{r-\beta}{\alpha})}(\log\varepsilon)^{\delta_{r,\beta}}\,,
\end{equation}
where $\delta$ denotes the Kronecker delta.

Finally, we can bound the expected overall cost:
\begin{align}
\cC^{\textnormal{CLMC}}_{L_{\max}} &= N\Ev\left[\int_0^{L\wedge L_{\max}}\cC(\ell)\,d\ell\right] \notag\\
&= N\int_0^{L_{\max}}\cC(\ell)\Py(L\ge \ell)\,d\ell \notag\\
&\le Nc_3 \int_0^{L_{\max}}e^{\gamma\ell}\,\Py(L\ge \ell)\,d\ell \notag\\[0.5ex]
&= \begin{cases} N\frac{c_3}{\gamma-r}\left(e^{(\gamma-r)L_{\max}}-1\right) &\mbox{if } r\ne \gamma\\[0.5ex]
Nc_3\gamma L_{\max} &\mbox{if } r = \gamma\,.
\end{cases}\label{totcostbound}
\end{align}
Hence, using 
\eqref{N_proof} the overall cost can be bounded as
\begin{equation}
\cC^{\textnormal{CLMC}}_{L_{\max}} \le C \, \varepsilon^{-2 -
                                     \max(0,\frac{r-\beta}{\alpha}) -
                                     \max(0,\frac{\gamma-r}{\alpha})}(\log\varepsilon)^{\delta_{r,\beta}+\delta_{r,\gamma}}\,, \label{complex_result_proof}
\end{equation}
for some constant $C>0$, which is again independent of $\varepsilon$. This
completes the proof since we had assumed that $r\in[\min(\beta,\gamma),
\max(\beta,\gamma)]$ and so $\max(0,\frac{r-\beta}{\alpha}) +
                                     \max(0,\frac{\gamma-r}{\alpha}) = \max(0,\frac{\gamma-\beta}{\alpha})$.

\end{proof}

\subsection{Proof of Corollary \ref{coroll_thm}}
\begin{proof}
To prove (a), suppose $L_{\max}=+\infty$. Then, the bias in
\eqref{biasbound} is zero due to Corollary \ref{coroll_unbiased}, so
that the MSE is equivalent to the
variance of the CLMC estimator. Since $r < \beta$ it follows as in the
proof of Theorem \ref{thm_CLMC_giles} in Section \ref{proof_thm}, that 
\[ \V\left[\widehat{Q}^{\textnormal{CLMC}}_{\infty}\right] \le \frac{C'}{N}\,, \]
for some constant $C'>0$.
Analogously, since $r>\gamma$, the expected overall cost can be bounded by
\[ C^{\textnormal{CLMC}}_{\infty} \le C'' N\,, \]
for some constant $C''>0$. Therefore, we can bound the MSE with
$\varepsilon^2$ by taking $N \ge C' \varepsilon^{-2}$ and the overall
computational cost is $C^{\textnormal{CLMC}}_{\infty} = \cO\left(\varepsilon^{-2}\right)$.

To prove (b), suppose that the additional assumptions in part $(b)$ of
Corollary \ref{coroll_thm} hold. Then, by tracking back the steps in
the proof of Theorem \ref{thm_CLMC_giles} in Section \ref{proof_thm},
it can be seen fairly easily that for $\beta\le \eta\le \gamma$ we have  
\[ \Ev\left[\V[\widehat{Q}^{\textnormal{CLMC}}_{\infty}|L]\right] \ge \begin{cases}
\frac{1}{N}\frac{4c_2'^2}{\eta(\eta-r)} &\mbox{if } r<\eta, r\ne \eta/2\\
\frac{1}{N}\frac{8c_2'^2}{\eta^2} &\mbox{if } r=\eta/2\\
+\infty &\mbox{if } r\ge \eta\,,
\end{cases} 
\quad \text{and} \quad
\cC^{\textnormal{CLMC}}_{\infty} \ge \begin{cases}
N\frac{c_3'}{r-\eta} &\mbox{if } r>\eta\\
+\infty &\mbox{if } r\le \eta\,.
\end{cases} \] 
We see that $\MSE \times \cC^{\textnormal{CLMC}}_{\infty}=+\infty$ for
all choices of $r$.

\end{proof}

\section{Goal-Oriented Error Estimators}\label{sec:error_estimator}

We use a classical goal-oriented error estimator to drive the
sample-wise adaptive scheme in our numerical experiments. The
following description is taken from \cite{Gra05}. Let $\omega
\in \Omega$ be fixed, and recall that $u \in V$ denotes the solution
of \eqref{eqn:variational_problem} whilst $u_h \in V_h \subset V$ is
its finite element approximation on a grid $\mathcal T_h$. The error
in a quantity of interest (defined by a linear
functional\footnote{Similar error estimators can also be obtained for
  nonlinear functionals by first linearising about $\epsilon_h$.}) is
given by
\begin{equation}
\cQ (\epsilon_h) =  \cQ(u - u_h) = \cQ(u) - Q(u_h).
\end{equation}
This functional can be interpreted as the `source' of the finite element discretisation error in the quantity
of interest, and is a bounded linear functional on the dual space
$V'$. The key idea of goal-oriented, a posteriori error estimators is to relate $\cQ(\epsilon_{h})$
to the solution residual $r^u_h$, i.e we seek a function $w  \in V''$ such that $\cQ(\epsilon_h) = w(r^u_h)$.
Since $V$ is a reflexive Hilbert Space,
there exists a $w \in V$ such that $\cQ(\epsilon_h) = r^u_h(w)$. The function $w$, termed the {\em influence
function}, is the solution of the {\em dual problem} \begin{equation}
a(v,w) = \cQ(v) \quad \forall v \in V.
\end{equation}
This dual solution can be approximate using the same finite element approximation as $u_h$, i.e. find
$w_h \in V_h \subset V$ s.t $$a(v_h,w_h) = \cQ(v_h)\quad \forall v_h
\in V_h\,.$$

Using the Galerkin orthogonality of
$u$ and $u_h$, we can bound  $\cQ (\epsilon_h)$ as follows:
\begin{align}
|\cQ(\epsilon_h) | &= |\cQ(u-u_h)| = |a(u-u_h,w)|
= |a(u-u_h,w)| + |a(u-u_h,w_h)| \nonumber \\[1ex]
&=|a(u-u_h,w - w_h)| \leq \sum\nolimits_{\tau \in \mathcal T_h} \| u -
u_h \|_{a,\tau} \| w - w_h \|_{a,\tau}\,.
\label{eqn:QOIupperbound}
\end{align}
In the last step, we have used the Cauchy-Schwarz inequality
elementwise. Hence, the product of energy norms $\| u -
u_h \|_{a,\tau} \| w - w_h \|_{a,\tau}$ provides an estimate for the
element-wise contribution to the error in $Q(u_h)$. It is
now used to define an appropriate adaptivity scheme.

To estimate the error of the solutions of the primal and dual problem
in the energy norm on each element $\tau$, we use
explicit error estimators. We only show the main ideas for estimating
$\|u-u_h\|_{a,\tau}$ using one
of the most basic estimators. The bound for
$\|w-w_h\|_{a,\tau}$ can be derived analogously.
On each element~$\tau$, using integration by parts, the FE error can 
be represented as
\begin{align}
a(\epsilon_{h},v)|_\tau &=  \int_\tau f v\; \td\textbf x - \int_\tau
                       \nabla u_{h} \cdot   k(\textbf x) \nabla v
                       \; \td \textbf x \nonumber\\
&= \int_\tau \mathcal R_u v\; \td \textbf x +
  \int_{\partial\tau} \mathcal J_u v\; \td s \quad \forall v \in V\,,\label{error_rep}
\end{align}
where the residual error on the element is define by
\begin{equation}
\mathcal R_u(\textbf x) = \nabla \cdot\textbf k(\textbf x)\nabla u_h
(\textbf x)  + f (\textbf x) \quad \forall  \textbf x \in \tau,
\end{equation}
and where $\mathcal J_u$ defines,  for all $\textbf x \in \partial \tau$ (except at
the vertices), the jump of the flux in $u_h$ across the element boundary by
\begin{equation}
    \mathcal J_u(\textbf x) =
    \begin{cases}
      k(\textbf x) \Big[\textbf n_\tau(\textbf x) \cdot \nabla
        u_h|_\tau + \textbf n_{\tau'(\textbf x)}(\textbf x) \cdot \nabla u_h|_{\tau'(\textbf x)}\Big], & \quad \forall
\textbf x \not\in \partial D\,,
\\[0.5ex]
\textbf n_\tau(\textbf x) \cdot k(\textbf x) \nabla
        u_h|_\tau \,,& \quad \forall \textbf x \in \partial D \,,\\
    \end{cases}
\end{equation}
where ${\bf n}_\tau$ is the outward unit normal to the element
boundary $\partial\tau$ at $\textbf x$ and $\tau'(\textbf x)$ is the
neighbouring element of $\tau$ at $\textbf x$. For simplicity, we
assume that the boundary conditions are homogeneous Dirichlet
conditions on all of $\partial D$. 

Using again Galerkin orthogonality, we can introduce the global FE
interpolant $\mathcal I_h v$ in \eqref{error_rep}, and thus using classical
interpolation theory find that
\begin{align*}
a(\epsilon_{h},v)|_\tau & \le \|\mathcal R_u\|_{L^2(\tau)}
                          \|v-\mathcal I_h v\|_{L^2(\tau)} +
\|\mathcal J_u\|_{L^2(\partial \tau)} \|v-\mathcal I_h
                          v\|_{L^2(\partial \tau)}\\
& \le c_1 \underbrace{\left( h_\tau \|\mathcal R_u\|_{L^2(\tau)}
+ \sqrt{h_\tau} \|\mathcal J_u\|_{L^2(\partial \tau)}
  \right)}_{\textstyle =: \eta_\tau(u_h)}
  \|v\|_{a,\omega_\tau} 
\,,
\end{align*}
where $\omega_{\tau}$ denotes the subdomain of elements sharing a
common edge with $\tau$, and where $c_1$ is problem
dependent constant independent of the mesh size
$h_\tau$. Substituting $v = \epsilon_h$ and summing over all elements, 
we can see that (up to a constant factor $c_2$ depending on the geometry)
this leads to the explicit global
energy error estimator
\begin{equation}
\|\epsilon_h\|_a \leq c_1c_2 \left(\sum\nolimits_{\tau \in \mathcal{T}_h}
\eta^2_\tau(u_h) \right)^{1/2}
\label{eqn:errorEstimator}
\end{equation}
for the primal solution on $\mathcal T_h$.

The local error contribution $\eta_\tau(w_h)$ to the dual solution
$w_h$ on $\tau$ in the energy norm can be estimated analogously, and
it can be shown that together with \eqref{eqn:QOIupperbound} this
leads to the goal-oriented error estimator 
\begin{equation}
|\mathcal Q(\epsilon_h)|\le c_3 \sum\nolimits_{\tau \in \mathcal{T}^{(k)}}
\eta_\tau(u_h) \eta_{\tau}(w_h) \,,
\label{eqn:GOest}
\end{equation}
which is again explicit up to the unknown constant $c_3$. 
 Although the exact constants in all the described estimators are not
known, the relative error with respect to a coarsest reference mesh
can still be used to drive a goal-oriented mesh adaptivity procedure, 
as described in Section \ref{sec:model}. 

More sophisticated error estimators exist, including estimators where
the constants are known or can be computed explicitly (see
e.g.~\cite{Gra05} for more details), but in our numerical experiments
we used the estimator described above.




\addcontentsline{toc}{section}{Bibliography}
\printbibliography
\nocite{*}

\end{document}

